\documentclass[]{interact}

\pdfoutput=1

\usepackage[numbers,sort&compress,merge]{natbib}

\usepackage{amsmath}
\usepackage{amssymb}
\usepackage{amsthm}
\usepackage{mathtools}
\usepackage{subcaption}
\usepackage{pgfplots}
\usepackage{tabularray}
\usepackage{appendix}
\usepackage{hyperref}
\usepackage{floatrow}

\usetikzlibrary{external}
\newcommand{\V}[1]{\textbf{#1}}

\theoremstyle{plain}
\newtheorem{theorem}{Theorem}[section]

\theoremstyle{definition}

\newtheorem{example}[theorem]{Example}
\newtheorem*{example*}{Example}

\theoremstyle{remark}
\newtheorem{remark}{Remark}
\newtheorem*{remark*}{Remark}

\bibpunct[, ]{[}{]}{,}{n}{,}{,}
\renewcommand\bibfont{\fontsize{10}{12}\selectfont}

\begin{document}

\title{Numerics with Coordinate Transforms for Efficient Brownian Dynamics Simulations}

\author{
\name{Dominic Phillips\textsuperscript{a}\thanks{CONTACT: \textsuperscript{a}dominic.phillips@ed.ac.uk, \textsuperscript{b}b.leimkuhler@ed.ac.uk, \textsuperscript{c}cmatthews@adobe.com}, Benedict Leimkuhler\textsuperscript{b} and Charles Matthews\textsuperscript{c}}
\affil{\textsuperscript{a}School of Informatics, University of Edinburgh; \textsuperscript{b}School of Mathematics, University of Edinburgh; \textsuperscript{c}Adobe Machine Learning}
}

\maketitle

\begin{abstract}
    Many stochastic processes in the physical and biological sciences can be modelled as Brownian dynamics with multiplicative noise. However, numerical integrators for these processes can lose accuracy or even fail to converge when the diffusion term is configuration-dependent. One remedy is to construct a transform to a constant-diffusion process and sample the transformed process instead. In this work, we explain how coordinate-based and time-rescaling-based transforms can be used either individually or in combination to map a general class of variable-diffusion Brownian motion processes into constant-diffusion ones. The transforms are invertible, thus allowing recovery of the original dynamics. We motivate our methodology using examples in one dimension before then considering multivariate diffusion processes. We illustrate the benefits of the transforms through numerical simulations, demonstrating how the right combination of integrator and transform can improve computational efficiency and the order of convergence to the invariant distribution. Notably, the transforms that we derive are applicable to a class of multibody, anisotropic Stokes-Einstein diffusion that has applications in biophysical modelling.
\end{abstract}

\begin{keywords}
Brownian Dynamics, 
Lamperti Transform, Time Rescaling,
Stochastic Differential Equations, Sampling 
\end{keywords}

\section{Introduction}
\label{sec:introduction}
Many problems in finance and the physical and biological sciences can be modelled as instances of Brownian dynamics. Examples include portfolio optimization \cite{merton_lifetime_1969}, options pricing \cite{black_pricing_1973}, diffusion in biological membranes and nanocomposites \cite{saffman_brownian_1975, fakhri_brownian_2010}, cell migration \cite{selmeczi_cell_2008}, protein folding \cite{best_coordinate-dependent_2010}, neuronal dynamics \cite{johnston_foundations_1995}, population genetics \cite{maruyama_stochastic_1983}, MRI imaging \cite{alexander_diffusion_2007}, ecological modelling \cite{varughese_statistical_2013} and score-based diffusion for generative AI \cite{song_score-based_2020}. In these contexts, configuration-dependent diffusion is often critical to the modelling assumption but it can introduce problems for numerical modelling. It can make the problem stiffer by introducing unbounded noise or bounds on the state variables. Additionally, it can reduce the weak order of convergence of an integrator. This is a problem for simulation because sampling becomes more expensive. It's also a problem for estimation, such as when fitting a Brownian dynamics `grey-box' model, since high accuracy is required for the Extended Kalman Filter approximations to be meaningful \cite{kristensen_parameter_2004}. 

One remedy for these problems is to design sophisticated, derivative-free numerical integrators that maintain high-accuracy convergence for certain classes of configuration-dependent diffusion. In recent years, many authors have contributed to a series of improvements and various integrators have been proposed \cite{milstein_stochastic_2004, rosler_rungekutta_2004, rosler_second_2009, lamba_adaptive_2007, anderson_weak_2010}. However, a common drawback of these integrators is the requirement of multiple evaluations of the force and diffusion tensor per time step. This can be prohibitively expensive for multi-body simulations, where the evaluation of these terms is the computational bottleneck \cite{sprinkle_brownian_2019}. Furthermore, many of these integrators place restrictions on the functional form of the diffusion, often requiring commutative noise, which is not suitable for all applications. 

An alternative approach, preferred whenever possible, is to transform the original process into a process with constant diffusion, thereby mitigating the sampling challenges introduced by multiplicative noise \cite{ait-sahalia_closed-form_2008}. For certain classes of stochastic differential equations (SDEs), this is achieved through a Lamperti transform, a type of non-linear change of state variables \cite{moller_state_2010, luschgy_functional_2006}. The resulting constant-diffusion process might exhibit enhanced numerical stability and can be sampled with computationally cheap, high weak-order integrators. Take for example the Black-Scholes model from financial mathematics, which describes geometric Brownian motion on the positive real axis. When simulated with sufficiently large step sizes, positivity can be violated which results in numerical instability. Here the Lamperti transform approach is especially valuable since it is possible to simultaneously construct a transform to unit diffusion whilst also removing the positivity constraint \cite{de_boer_lamperti_2020}. 

An alternative to a spatial coordinate transform is to apply a smooth, configuration-dependent time-rescaling \cite{oksendal_stochastic_2003, agazzi_introduction_2023}. 
Recently, this has been explored as a method for adaptive step size control in Langevin dynamics sampling \cite{leroy_adaptive_2024}. In this work, we take a different perspective and consider time-rescaling alongside the Lamperti transform as a strategy to remove multiplicative noise. 

In this article, we derive conditions for applying the Lamperti and time-rescaling transforms, either separately or in combination, to achieve constant diffusion in multivariate Brownian dynamics with multiplicative noise. Through numerical experiments, we show how if the right choice of numerical integrator is used for the transformed process, then this leads to an efficient, second-order weak sampling method that involves just one force and one diffusion evaluation per time step. Furthermore, we show how the original autocorrelation function and evolving distribution can be accurately recovered by applying an inverse transform to the samples. We also draw connections with rare event sampling, illustrating by means of a one-dimensional example how coordinate transforms can be used in that context.

The article is structured as follows. Section \ref{sec:preliminaries} introduces Brownian dynamics and the Lamperti and time-rescaling transformations. Section \ref{sec:OneDimensionTransforms} explores in detail how these transforms apply to one-dimensional Brownian dynamics, including rare event sampling. Section \ref{sec:MultivariateTheory} extends the theory of transforms to multivariate Brownian dynamics. Numerical experiments in one dimension are presented in Section \ref{sec:NumericalExp1D} and multivariate experiments are presented in Section \ref{sec:MultivariateExperiments}. Conclusions are presented in Section \ref{sec:Conclusions}.

\section{Preliminaries}
\label{sec:preliminaries}
\subsection{Brownian dynamics}

Brownian dynamics is defined through an It\^o stochastic differential equation (SDE), which in one dimension reads \cite{pavliotis_stochastic_2014}
\begin{equation}
\label{1Dvariable2}
    dx_t = - D(x_t) \frac{dV(x_t)}{dx}dt + kT \frac{d D(x_t)}{dx}dt + \sqrt{2 kT D(x_t)} dW_t,
\end{equation}
where $t \in \mathbb{R}_{>0}$ is time, $x_t \in \mathbb{R}$ is the state variable, $W_t$ is a one-dimensional Wiener process, $V : \mathbb{R} \xrightarrow{} \mathbb{R}$ is a potential energy function, $D : \mathbb{R} \xrightarrow{} \mathbb{R}_{>0}$ is the diffusion coefficient, $k$ is the Boltzmann constant and $T$ is the temperature in degrees Kelvin. Note that the diffusion coefficient $D(x)$ is a function of $x$ which means that we have configuration-dependent noise, also known as multiplicative noise. 

In higher dimensions, \eqref{1Dvariable2} generalises to
\begin{equation}
\label{multivariateBrownianDiffusion}
d\V{X}_t = -(\V{D}(\V{X}_t)\V{D}(\V{X}_t)^T) \nabla V(\V{X}_t) dt + kT \text{div}(\V{D}\V{D}^T)(\V{X}_t) dt + \sqrt{2 kT} \V{D}(\V{X}_t)d\V{W}_t,
\end{equation}
where $\V{X}_t \in \mathbb{R}^n$ is the state variable, $\V{W}_t$ is an n-dimensional Wiener process, $V: \mathbb{R}^n \xrightarrow{} \mathbb{R}$ is a potential function, and $\V{D}\V{D}^T: \mathbb{R}^n \xrightarrow{} \mathbb{R}^n \times \mathbb{R}^n$ is a configuration-dependent diffusion tensor that is everywhere positive definite. The matrix divergence in \eqref{multivariateBrownianDiffusion} is defined as the column vector resulting from applying the vector divergence to each matrix row. We identify $\V{DD}^T$ (not $\V{D}$) as the diffusion tensor to avoid taking a matrix square root in the noise term.

We assume that $V$ is confining in a way that ensures ergodicity of the dynamics. This is true for any $V$ that grows sufficiently quickly as $|\V{X}| \xrightarrow[]{} \infty$, for details see Pavliotis (2014) \cite{pavliotis_stochastic_2014}. One consequence of ergodicity is that there exists a unique invariant distribution $\rho(\V{X})$ - a probability distribution that does not change under the process dynamics. For Brownian dynamics, the invariant distribution is the canonical ensemble; $\rho(\V{X}) \propto \exp{\left(- V(\V{X})/kT\right)}$. Another consequence of ergodicity is that the long-time time average of any L1 integrable function $f$ converges to its phase-space average as the simulation time goes to infinity, i.e.
\begin{align}
\label{eqn:theErgodicTheorem}
\int_{\mathbb{R}^n} f(\V{X})\rho(\V{X})d\V{X} = \lim_{T_{\text{sim}} \rightarrow \infty} \frac{1}{T_{\text{sim}}} \int_{t=0}^{T_{\text{sim}}} f(\V{X}_t)dt.
\end{align}
In the remainder of this paper, we shall refer to \eqref{eqn:theErgodicTheorem} as `the ergodic theorem'. 

\subsection{The Lamperti transformation}

\label{lampertiTransformTheory}
Consider a time-homogeneous Itô SDE of the form

\begin{equation}
\label{generalSDE}
d\V{X}_t = f(\V{X}_t)dt + \mathbf{\sigma} (\V{X}_t)\V{R}d\V{W}_t,
\end{equation}
\noindent
where $\V{X}_t \in \mathbb{R}^n$ is the state vector, $\V{W}_t \in \mathbb{R}^m$ is an $m$-dimensional Wiener process, $f : \mathbb{R}^n \xrightarrow{} \mathbb{R}^n$ is a drift function, $\sigma : \mathbb{R}^n \xrightarrow{} \mathbb{R}^n \times \mathbb{R}^m$ is a configuration-dependent diffusion matrix and $\V{R} \in \mathbb{R}^m \times \mathbb{R}^m$ is a constant matrix. 

The \textit{Lamperti transform} \cite{moller_state_2010}, $\xi \vcentcolon \mathbb{R}^n \xrightarrow{} \mathbb{R}^n$ is an invertible coordinate transformation that when applied to \eqref{generalSDE}, results in a process $\V{Y}_t \vcentcolon= \xi(\V{X}_t)$ having unit diffusion. The required transform can be derived by applying the multivariate version of It\^o's lemma and setting the coefficients of the noise terms to one. This gives a set of ODEs that $\xi$ must satisfy. A consistent solution exists if: (i) the dimensions of the state variable and the noise are the same, (ii) $\V{R}$ is invertible, (iii) $\sigma(\V{X}_t)$ has the diagonal form:
\begin{equation}
\label{sigmaRestrictionLamperti}
\sigma(\V{X}_t) = \text{diag}(\sigma_1(X_{1,t}), \sigma_2(X_{2,t}), \dots, \sigma_n(X_{n,t})),
\end{equation}
where $\sigma_i : \mathbb{R} \xrightarrow{} \mathbb{R}_{>0}$ for all $i \in \{1,2,\dots,n\}$. The solution is given by
\begin{equation}
\label{multivariateTransformedLampertiProcess}
    \V{Y}_t = \xi(\V{X}_t) = \V{R}^{-1}\phi(\V{X}_t),
\end{equation}
where $\phi(\V{X}_t) =  [\phi_1(X_{1,t}), \phi_2(X_{2,t}), \dots, \phi_n(X_{n,t})]^T$ and $\phi_j : \mathbb{R} \xrightarrow{} \mathbb{R}$ is the invertible function:
\begin{equation}
\label{psi_j}
\phi_j(x) \vcentcolon= \int_{x_{j,0}}^{x} \frac{1}{\sigma_j(z)} dz,
\end{equation}
with $x_{j,0}$ being an arbitrary constant chosen from the state space of $X_{j}$. The governing equation of $\V{Y}_t$ then reads:
\[
dY_{i,t} = \sum_{j=1}^n R_{ij}^{-1} \left(\frac{f_j(\phi^{-1}(\V{R}\V{Y}_t))}{\sigma_j(\phi^{-1}_j((\V{R}\V{Y}_t)_j))} - \frac{1}{2}\frac{\partial}{\partial x} \sigma_j \left(x\right)\Bigg|_{x=\phi^{-1}_j((\V{R}\V{Y}_t)_j)}\right)dt + dW_{i,t}.
\]

The Lamperti transform can be used as a tool to find exact solutions for specific classes of SDEs \cite{de_boer_lamperti_2020} or to perform statistical inference for SDEs \cite{craigmile_statistical_2023}, but the extent to which the Lamperti transform is useful in practice is limited by the requirement in \eqref{sigmaRestrictionLamperti}. We consider only time-homogeneous SDEs in this paper, although the Lamperti transform can also be extended to certain time-inhomogeneous problems \cite{moller_state_2010}.

\subsection{The time-rescaling transform}
\label{timeRescalingTheory}
An alternative approach for transforming an SDE to constant diffusion is the time-rescaling transformation (see, for instance, \cite{oksendal_stochastic_2003} Chapter 8 and \cite{agazzi_introduction_2023} Chapter 8). As before, consider an SDE of the form in Equation \eqref{generalSDE}. We introduce a configuration-dependent time rescaling, denoted as $t \rightarrow \tau(t)$, with Jacobian $\frac{dt}{d\tau}(\V{X}_t)=g(\V{X}_{\tau})$. The governing equation for the time-rescaled process becomes
\begin{equation}
\label{timeRescaledEqnsBeforeSub}
    d\V{X}_{\tau} = f(\V{X}_\tau)g\left(\V{X}_{\tau}\right)d\tau + \mathbf{\sigma} (\V{X}_\tau)\V{R}\sqrt{g\left(\V{X}_{\tau}\right)}d\V{W}_\tau,
\end{equation}
where we have replaced $dt$ with $\frac{dt}{d\tau}  d\tau = g(\mathbf{X}_\tau)d\tau$ using a change of variables. The factor $\sqrt{g(\mathbf{X}_\tau)}$ in the noise arises from the scaling property of Brownian motion. 

A transformation to unit diffusion is possible if and only if: (i) the dimensions of the state variable and the noise are the same, (ii) $\V{R}$ is invertible, (iii) the diffusion matrix $\sigma(\V{X}_t)$ has the diagonal form:
\begin{equation}
\label{eqn:timeRescalingAnsatz}
    \sigma(\V{X}_t) = \text{diag}(D(\V{X}_t), D(\V{X}_t), \dots, D(\V{X}_t)),
\end{equation}
an isotropic matrix with arbitrary configuration dependence. 

To remove the configuration dependence from the diffusion term, we choose $g(\V{X}) = 1/D^2(\V{X})$. Substituting this and the isotropic ansatz \eqref{eqn:timeRescalingAnsatz} into \eqref{timeRescaledEqnsBeforeSub} simplifies the governing equations to
\[
\label{eqn:timeRescalingFinal}
    d\V{X}_{\tau} = \frac{f(\V{X}_\tau)}{D^2(\V{X}_{\tau})}d\tau + \V{R}d\V{W}_\tau.
\]
We may then transform to unit diffusion through a linear transform
\[
    \label{eqn:timeRescalingLinearTransform}
    \V{Y}_{\tau} = \V{R}^{-1} \V{X}_{\tau}.
\]
Note that time-rescaling method can also be used to transform to an arbitrary isotropic diffusion $\tilde{D}(\mathbf{X})$ by making the choice $g(\V{X})=(\tilde{D}(\V{X})/D(\V{X}))^2$.

\section{Transforms for one-dimensional Brownian dynamics}
\label{sec:OneDimensionTransforms}
In this section, we consider the Lamperti and time-rescaling transforms applied to one-dimensional Brownian dynamics, comparing the two approaches. For detailed proofs of all results, see Appendix \ref{appdx:Proofs}.

\subsection{The Lamperti transform} 

In one dimension, the Lamperti transform emerges as an instance of a transformational symmetry inherent in Brownian dynamics. This symmetry states that, under an invertible coordinate transformation $x \xrightarrow{} y(x)$, the one-dimensional Brownian dynamics \eqref{1Dvariable2} with potential $V(x)$ and diffusion function $D(x)$ is transformed into another Brownian dynamics process with potential $\hat{V}(y)$ and diffusion function $\hat{D}(y)$ given by
\begin{equation}
\label{eqn:lamperti1DPotentialAndD}
\begin{split}
    \hat{V}(y) &= V(x(y)) + kT \ln \left\vert \frac{dy}{dx}(x(y)) \right\vert,  \\
    \hat{D}(y) &= {D}(x(y))\left( \frac{dy}{dx}(x(y))\right)^2,
\end{split}
\end{equation}
where $y \xrightarrow{} x(y)$ is the inverse transformation (Theorem \ref{thm:transformationalSymmetry}). 

\noindent
Setting $\hat{D}(y)=1$ and solving for $y(x)$ yields the one-dimensional Lamperti transform
\begin{equation}
\label{eqn:Lamperti1DIntegral}
y(x) = \int_{x_0}^{x} \left( \frac{1}{{D}(z)}\right)^{\frac{1}{2}}dz,
\end{equation}
which is equation \eqref{psi_j} with $\sigma(z) = \sqrt{D(z)}$. 
\noindent
From \eqref{eqn:Lamperti1DIntegral} we have $\frac{dy}{dx} = \left( \frac{1}{{D}(x)}\right)^{\frac{1}{2}}$. Substituting this result into \eqref{eqn:lamperti1DPotentialAndD}, we arrive at the transformed, constant-diffusion dynamics:
\begin{equation}
\label{eqn:1DLampertiTransform}
dy_t = - \frac{d\hat{V}(y_t)}{dy}dt + \sqrt{2kT}dW_t, 
\end{equation}
\noindent
with an effective potential given by
\[
\hat{V}(y) = V(x(y)) - \frac{kT}{2} \ln{D(x(y))}.
\]
Note that $\hat{V}(y)$ implicitly depends on $x_0$ in \eqref{eqn:Lamperti1DIntegral} through the inverse transform $x(y)$. Since $x_0$ changes the vertical offset of $y(x)$, it therefore changes the horizontal offset of $x(y)$. Changing $x_0$ thus corresponds to horizontally translating $\hat{V}(y)$, which shifts the mean position but otherwise has no physical consequence for the dynamics.

By writing down the ergodic theorem for the process \eqref{eqn:1DLampertiTransform} and transforming back to $x$-space, it can be shown that (Theorem \ref{thm:LampertiPhaseSpace1D})
\begin{equation}
\label{eqn:1DLampertiErgodicTheorem}
\int_{-\infty}^{\infty} f(x)\rho(x)dx =  \lim_{T_{\text{sim}} \rightarrow \infty} \frac{1}{T_{\text{sim}}} \int_{t=0}^{T_{\text{sim}}} f(x(y_t)) dt,
\end{equation}
\noindent
so trajectories of the transformed process can be used directly to approximate ensemble averages with respect to $\rho(x)$, the invariant distribution of the original process. Furthermore, by choosing $f(x)=I(x \in [a,b])$ (the indicator function on the interval $[a,b]$) it can be shown that invariant distribution $\rho(x)$ of the original process and the invariant distribution $\hat{\rho}(y)$ of the Lamperti-transformed process are related by the equation $\rho(x) = \hat{\rho}(x(y))\frac{dy}{dx}$ (Theorem \ref{thm:InvariantLamperti}). Similarly, if we have a trajectory of discrete samples $y_n$ with constant step size $h$, then choosing $f(x)=I(x \in [a,b])$ in \eqref{eqn:1DLampertiErgodicTheorem} leads to a simple counting formula to estimate finite-width integrals of the original invariant distribution:

\begin{equation}
\label{reweightingLamperti}
\int_a^b \rho(x) dx \approx \lim_{N \rightarrow \infty} \frac{1}{N}\sum_{n=0}^N I(x(y_n) \in [a,b]).
\end{equation}
This approximation becomes exact in the limit $h \xrightarrow{} 0$. 

\subsection{The time-rescaling transform}

Consider a configuration-dependent time rescaling $t \xrightarrow{} \tau(t)$ with $\frac{dt}{d\tau}(x) = g(x)$. It can be shown that applying this to the original Brownian dynamics \eqref{1Dvariable2} results in another Brownian dynamics process but with a modified potential $\hat{V}(x)$ and diffusion coefficient $\hat{D}(x)$, given by (Theorem \ref{thm:TimeRescaling1D}):
\begin{equation}
\label{eqn:effectivePotentialTimeRescaling}
\begin{split}
    \hat{V}(x) &= V(x) + kT \ln{g(x)}, \\
    \hat{D}(x) &= g(x)D(x).
\end{split}
\end{equation}
\noindent
Setting $\hat{D}(x)=1$ implies $g(x) = \frac{1}{D(x)}$. Substituting this result into \eqref{eqn:effectivePotentialTimeRescaling}, we arrive at the time-rescaled, constant-diffusion dynamics:
\[
    dx_\tau = - \frac{d\hat{V}(x_\tau)}{dx}d\tau + \sqrt{2kT}dW_\tau,
\]
\noindent
where
\[
 \hat{V}(x) = V(x) - kT \ln{D(x)}    
\]
\noindent
is the effective potential. Notably, these dynamics differ from those obtained through the Lamperti transform. 

By applying a time rescaling to the ergodic theorem of the original process $x_t$, it can be shown that (Theorem \ref{thm:TimeRescalingPhaseSpace1D}):
\begin{equation}
\label{reweightingTimeRescaling0}
\int_{-\infty}^{\infty} f(x) \rho(x) dx = \lim_{\mathcal{T}_{\text{sim}} \rightarrow \infty} \frac{\int_{\tau=0}^{\mathcal{T}_{\text{sim}}} f(x_\tau) g(x_\tau) d\tau}{\int_{\tau=0}^{\mathcal{T}_{\text{sim}}} g(x_\tau) d\tau},
\end{equation}
so trajectories of the transformed process can be used directly to approximate ensemble averages with respect to $\rho(x)$, the invariant distribution of the original process.

Discritising with a constant step size $h$ in $\tau$-time, and setting $f(x) = I(x \in [a,b])$, leads to a counting formula to estimate finite-width integrals of the original invariant distribution: 
\begin{equation}
\label{reweightingTimeRescaling}
    \int_{a}^{b} \rho(x) dx \approx  \lim_{N \rightarrow \infty} \frac{\sum_{n=0}^{N} g(x_{\tau_n})I(x_{\tau_n} \in [a, b])}{\sum_{n=0}^N g(x_{\tau_n})}.
\end{equation}
This approximation becomes exact in the limit $h \xrightarrow{} 0$.

\begin{remark*}
The proof of \eqref{reweightingTimeRescaling0} and \eqref{reweightingTimeRescaling} does not require Brownian dynamics, hence these results hold more generally for one-dimensional, time-homogeneous SDEs.
\end{remark*}

\subsection{Comparing the two transform approaches}
\label{sec:comparingApproaches}

\begin{figure}
    \centering
    \includegraphics[scale=0.45]{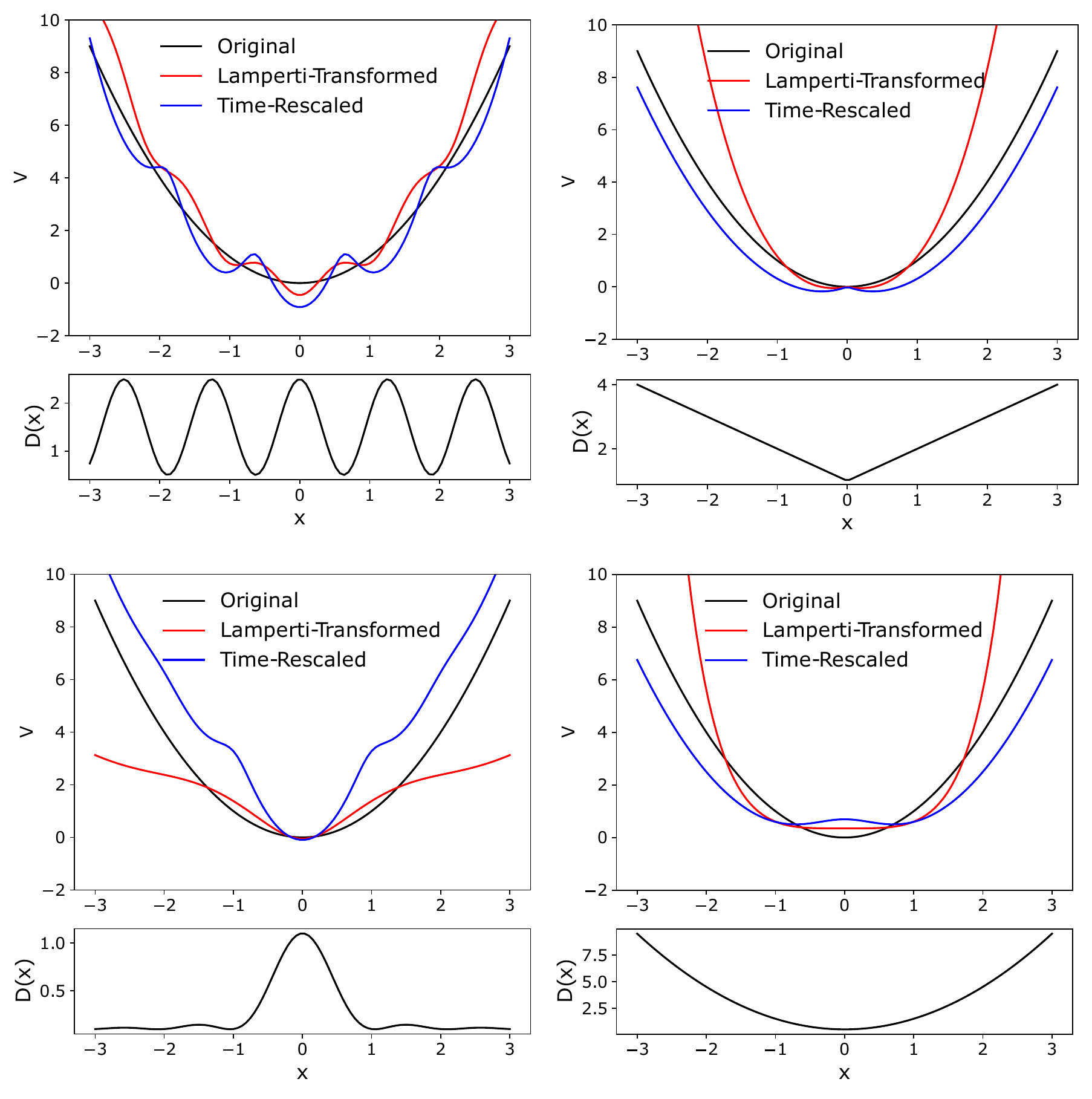}
    \caption{Comparison of the Lamperti and time-rescaling transforms when applied to the same quadratic potential $V(x)=x^2$ for a variety of different diffusion coefficients. The abscissa axis represents the original $x$ coordinate for the time-rescaled potential and is the transformed $y(x)$ coordinate for the Lamperti-transformed potential.}
    \label{fig:comparingTransforms}
\end{figure}


In one dimension, both the Lamperti and time-rescaling transforms are applicable for any $D(x) > 0$. However, whilst the time-rescaling transform can be computed exactly, the Lamperti transform often requires a numerical approximation due to the intractability of the integral \eqref{eqn:Lamperti1DIntegral}. The two transforms also yield different effective potentials. The Lamperti transform tends to increase confinement of the potential in $y$-space wherever $D(x) > 1$ and decrease it wherever $0 < D(x) < 1$, whereas the time-rescaled effective potential is more confining where $\frac{dD}{dx} < 0$ and less confining where $\frac{dD}{dx} > 0$. 
Therefore, deciding which transform is more useful can depend on the functional form of $D(x)$. For rare-event sampling, one should choose whichever transform results in the least-confining effective potential as this improves numerical stability for larger step sizes. Figure \ref{fig:comparingTransforms} compares the effective potentials resulting from the two transforms for various initial diffusion coefficients.  In the next section, we provide a deeper insight into coordinate transforms for rare event sampling. We then derive multivariate versions of the Lamperti and time-rescaling transforms in Section \ref{sec:MultivariateTheory}.

\subsection{Connection to Variable-Diffusion Enhanced Sampling}
\label{sec:EnhancedSampling}

Recall that in this work, we generally consider that $D(x)$ is fixed by the application problem. Due to it being non-constant, it leads to less efficient numerical computations. Hence, the motivation to ``transform it away" into constant noise. Nonetheless, for problems where the primary goal is to explore a potential landscape, as is common in biomolecular modelling \cite{abrams_enhanced_2014}, it can be beneficial to take a different perspective and consider $D(x)$ as a free parameter that can be optimised. For this to be successful, the drawback of corrupting the dynamics must be offset by kinetic benefits, in the form of reduced metastability, gained by modifying the diffusion. A prototype of this idea dates back to the work of Roberts and Stramer (2002), what the authors call \textit{Langevin Tempered Diffusion} \cite{roberts_langevin_2002}. More recently, the optimisation problem for $D(x)$ was formulated precisely by Lelièvre et al. for dynamics without coordinate transforms \cite{lelievre_optimizing_2024}. Work on the Lamperti transform for finding natural reaction coordinates in small proteins was explored in Krivov and Karplus (2008) \cite{krivov_diffusive_2008}. The theoretical results in the previous section also offer new intuition for this class of sampling problem. We illustrate this through a numerical example below. For concreteness, we consider only the time-rescaling transform. Similar results could be obtained using the Lamperti transform.

\begin{example}
\textbf{Enhanced Sampling in a Double Well} 

\begin{figure}
  \centering
  \begin{subfigure}{0.49\textwidth}
    \centering
    \includegraphics[width=\linewidth]{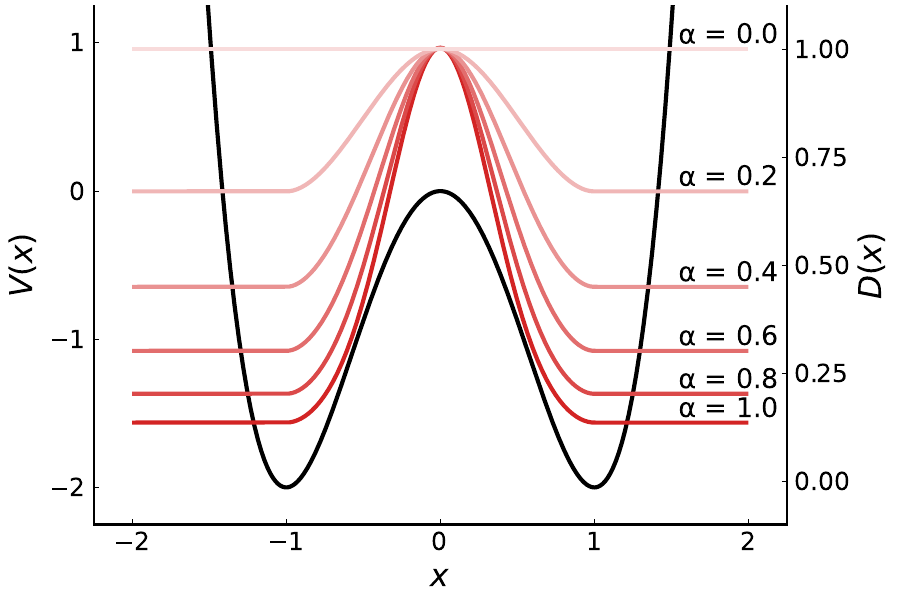}
    \caption{}
  \end{subfigure}
  \hfill
  \begin{subfigure}{0.49\textwidth}
    \centering
\includegraphics[width=\linewidth]{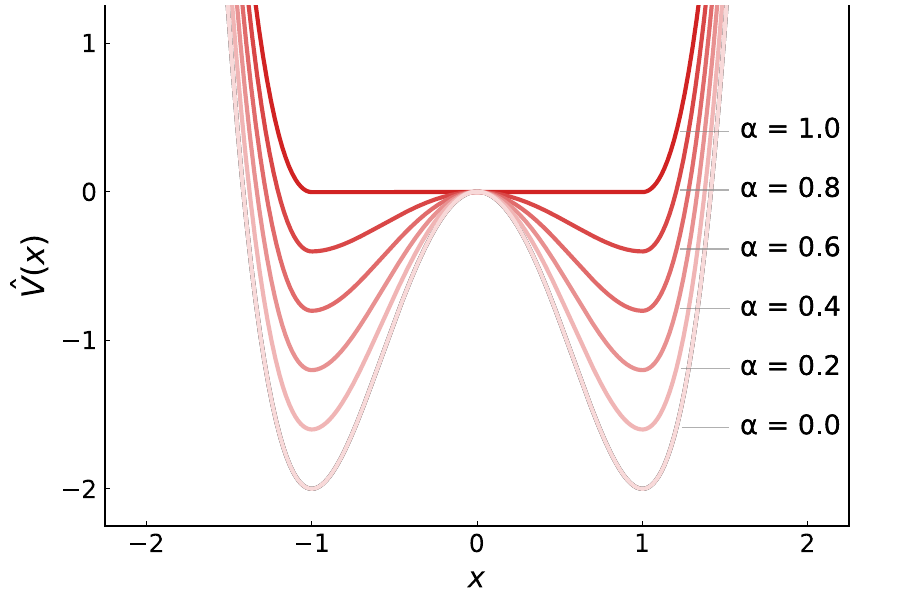}
    \caption{}
  \end{subfigure}
  \caption{(a) Black: the double-well potential $V(x)$. Red: the configuration-dependent diffusion $D(x; \alpha)$ for various $\alpha$ with $kT=1$. (b) Time-rescaled potentials after transforming to constant diffusion $D=1$ with initial diffusion $D(x; \alpha)$. By design, the time-rescaled potentials are independent of $kT$. The transform has the effect of reducing the metastability of the well, removing it entirely when $\alpha = 1$.}
\label{fig:doubleWellPlot}
\end{figure} 

\begin{figure}
    \centering
    \includegraphics[scale=0.65]{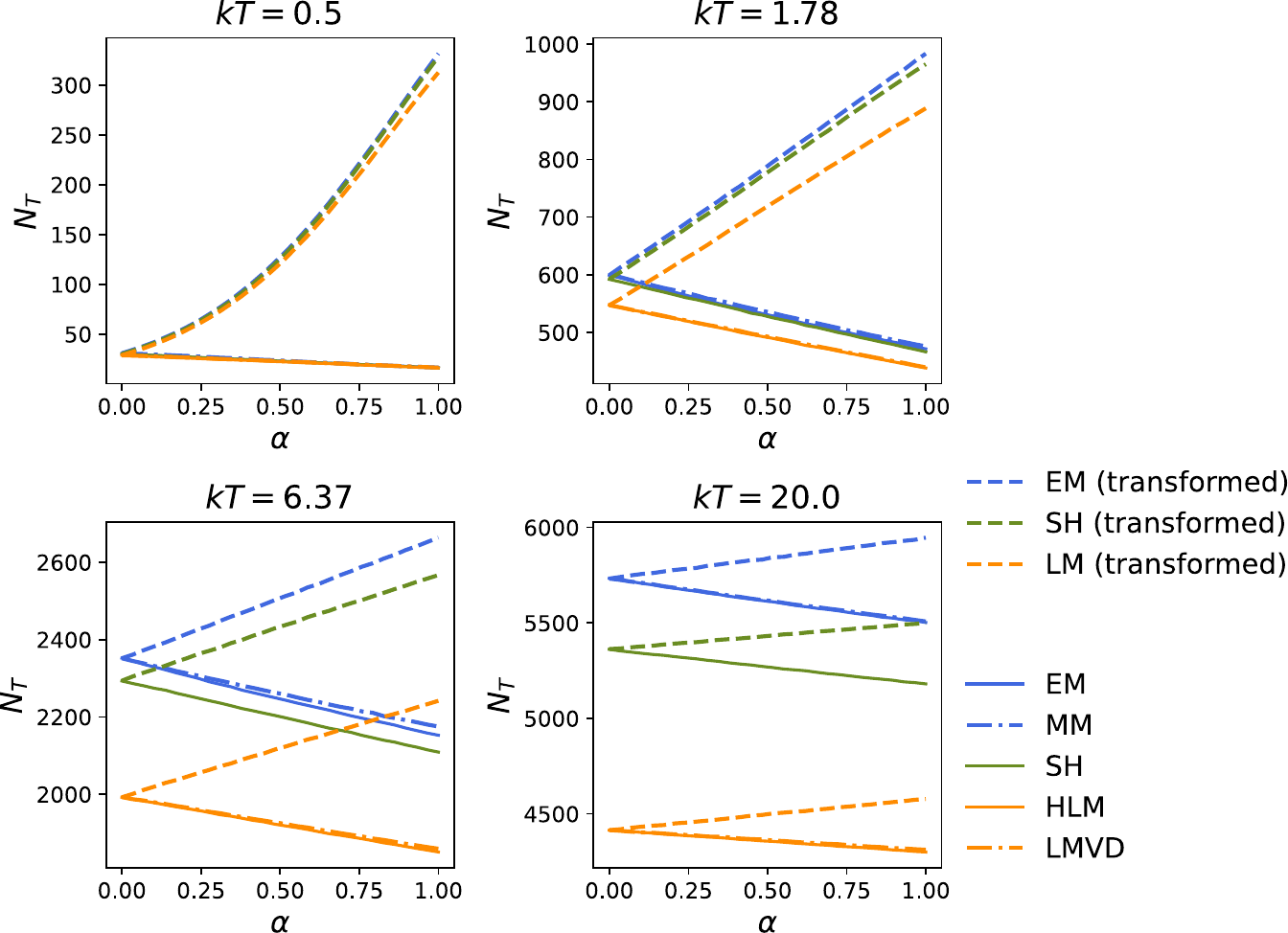}
    \caption{The mean transition counts ($N_T$) vs $\alpha$ for Brownian dynamics trajectories of length $T=1000$ in a double-well potential with a constant step size $h=0.01$. Each line represents an average over $5 \times 10^3$ independent repeats. Results are shown for four different values of $kT$, equally spaced in logarithmic scale. Solid lines represent integrators for the untransformed potential with diffusion function $D(x; \alpha)$, while dotted lines represent integrators for the time-transformed potential with $D(x) = 1$.}
    \label{fig:transitionCounts}
\end{figure}

\begin{figure}
    \centering
    \includegraphics[scale=0.65]{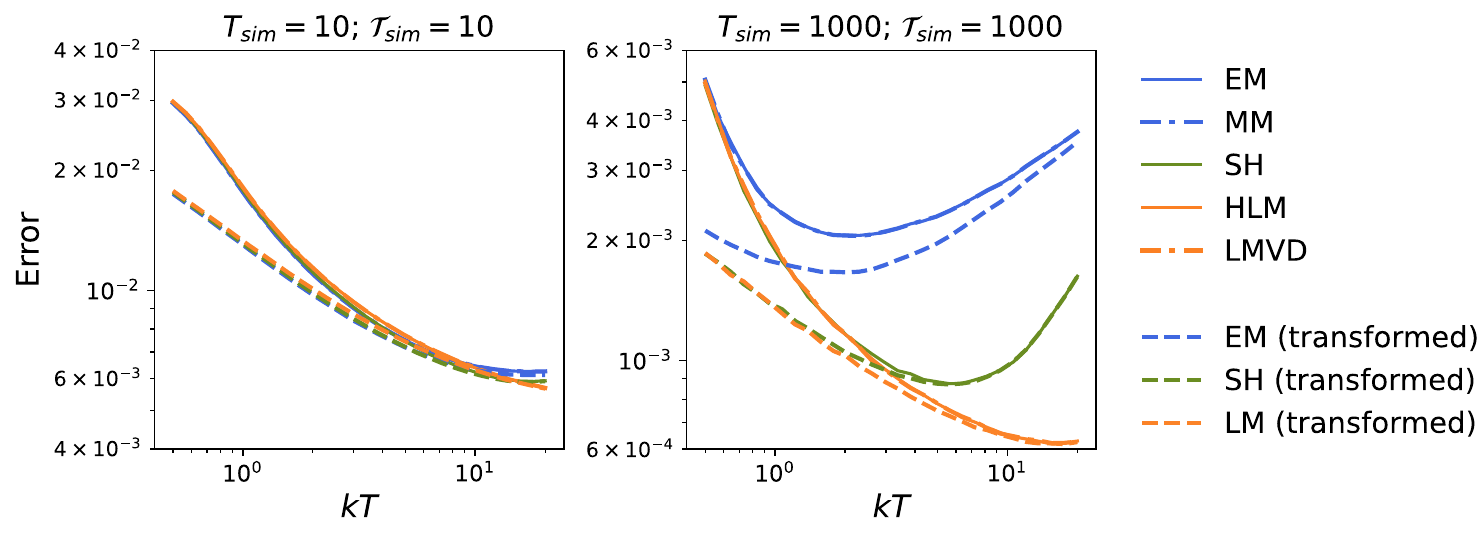}
    \caption{$L_1$ error in the invariant measure as a function of $kT$ for the different integrators. Solid lines correspond to integrators of the original dynamics with $\alpha = 0$ (constant, unit diffusion). Dotted lined correspond to integrators of the transformed dynamics using the $\alpha$ value that gives the minimum $L_1$ error at that temperature. The left panel shows the error for trajectories to final time $10$, where Monte Carlo error generally dominates. The right panel shows the error for trajectories to final time $1000$, where discretisation error generally dominates.}
    \label{fig:L1_error_compared}
\end{figure}

\noindent
Consider the double-well potential 
\begin{equation}
V(x) = - \frac{1}{4}h^4 x^2 + \frac{1}{2}c^2 x^4
\end{equation}
with $h=c=2$. This has a local maximum at $(x_{\text{max}}, V_{\text{max}})=(0,0)$, minima at $(x_{\pm}, V_{\text{min}})=(\pm 1, -2)$ and inflection points at $(\pm \frac{\sqrt{3}}{3},-\frac{11}{9})$. Using this potential, we define a family of configuration-dependent diffusion functions:
\begin{equation}
\label{eqn:ScaledDiffusion}
D(x; \alpha) = \begin{cases}
\exp\left(\alpha \cdot \frac{V(x)}{kT}\right) \quad \text{if } x_{-}< x < x_{+}\\
\exp\left(\alpha \cdot \frac{V_{\text{min}}}{kT}\right) \quad \; \text{otherwise},
\end{cases}
\end{equation}
where $0 \leq \alpha \leq 1$ controls the degree of configuration dependence, $\alpha = 0$ corresponds to constant noise. In Figure \ref{fig:doubleWellPlot}(a) we plot the potential and diffusion functions for $x \in [-2, 2]$ and $\alpha \in \{0.0, 0.2, 0.4, 0.6, 0.8, 1.0\}$ with $kT = 1$. Transforming \eqref{eqn:ScaledDiffusion} to constant diffusion using a time-rescaling leads to an effective potential
\begin{equation}
\label{eqn:timeRescaledPotential}
\hat{V}(x; \alpha) = \begin{cases}
    \left(1 - \alpha\right) V(x) \quad \quad \text{if } x_{-}< x < x_{+}\\
    V(x) + \alpha \vert V_{\text{min}} \vert \quad \text{otherwise}.
\end{cases}
\end{equation}

We plot this effective potential in Figure \ref{fig:doubleWellPlot}(b). As $\alpha$ increases, the rescaled potential flattens, gradually removing the metastability. This smooth removal of the metastability motivates the choice of diffusion considered in \eqref{eqn:ScaledDiffusion}.

In Figure \ref{fig:transitionCounts} we plot the $\alpha$ and $kT$ dependence of the mean transition counts, $N_T$, between the wells for various numerical integrators (for details of the integrators, see Section \ref{sec:numericalIntegrators}). Unsurprisingly, simulations of the time-rescaled dynamics with $\alpha > 0$ result in enhanced sampling between the two wells, due to the lower energetic barrier. Furthermore, this enhanced-sampling effect is strongest at lower temperatures, since diffusion across the barrier is a kinetically-inhibited rare event. 

For fixed step size, numerical integrator, and temperature, there exists an $\alpha$ value for which the sampling error of the transformed dynamics is minimum, bounded above by the sampling error with untransformed dynamics ($\alpha = 0$). In Figure \ref{fig:L1_error_compared} we plot this minimum sampling error as a function of $kT$ with step size $h = 0.01$ for short simulations ($T_{\text{sim}}, \mathcal{T}_{\text{sim}} =10$, Monte-Carlo error dominates the sampling error) and long simulations ($T_{\text{sim}}, \mathcal{T}_{\text{sim}} =1000$, discretisation error dominates). For all integrators, sampling efficiency can be improved by considering an $\alpha >0$ and applying a time rescaling to constant diffusion. Once more, we observe that the effect is most dramatic at low temperatures. See Appendix \ref{appx:doubleWellEnhancedSampling} for the full experimental methodology.

\end{example}

\section{Transforms for multivariate Brownian dynamics}
\label{sec:MultivariateTheory}
This section examines generalisations of the Lamperti and time-rescaling transforms to multivariate Brownian dynamics. Proofs of all results can be found in Appendix \ref{appdx:Proofs}.

\subsection{The multivariate Lamperti transform}
\label{sec:multivariateLamperti}

Consider multivariate Brownian dynamics with $\V{D}$ matrix
\[
\V{D}(\V{X})_{ij} = D_i(X_i)R_{ij},
\]
where $R_{ij}$ in an invertible, constant matrix. For this class of diffusion, although a Lamperti transform to unit diffusion can be constructed (Section \ref{lampertiTransformTheory}), the transformed dynamics are only Brownian dynamics if $\V{R}$ is proportional to the identity. Specifically, when $\V{D}(\V{X})_{ij} = D_i(X_i)\delta_{ij}$, the Lamperti-transformed process is $Y_{i,t} = \sqrt{2kT} \int_{x_0}^{X_{i,t}} \frac{1}{D_i(x)} dx \vcentcolon= \sqrt{2kT} \phi_i(X_{i,t})$, and obeys
\[
dY_{i,t} = -\nabla_{Y_i} \hat{V}({\V{Y}})dt + \sqrt{2kT}dW_i,
\]
with an effective potential (Theorem \ref{thm:MultivariateLamperti})
\begin{equation}
\label{eqn:LampertiMultivariateEffectivePotential}
    \hat{V}(\V{Y}) = V(\phi^{-1}(\V{Y})) - kT \sum_{k=1}^n \ln D_k (\phi^{-1}_k (Y_{k,t})).
\end{equation}
In this case, the ergodic theorem generalises to (Theorem \ref{thm:LampertiErgodic})
\begin{equation}
\label{eqn:NDLampertiErgodicTheorem}
\int_{\mathbb{R}^n} f(\V{X})\rho(\V{X}) d\V{X} = \lim_{T_{\text{sim}} \rightarrow \infty} \frac{1}{T_{\text{sim}}} \int_{t=0}^{T_{\text{sim}}} f(\phi^{-1}(\V{Y}_t)) dt,
\end{equation}
where the map $\phi^{-1}: \mathbb{R}^n \rightarrow \mathbb{R}$ is constructed by individually applying $\phi_i^{-1}$ to each component of its argument, $1 \leq i \leq n$. We observe that there is an independent contribution to the effective potential for every diagonal component of $\V{D}$.

\subsection{The multivariate time-rescaling transform}
\label{sec:multivariateTimeRescaling}

Consider multivariate Brownian dynamics with $\V{D}$ matrix
\[
\V{D}(\V{X}) = D(\V{X})\V{R},
\]
where $\V{R}$ is invertible. For this class of $\mathbf{D}$ matrix, a time-rescaling to Brownian dynamics unit diffusion can be constructed (Section \ref{timeRescalingTheory}). The time-rescaled process is given by $\V{Y}_\tau = \mathbf{R}^{-1}\V{X}_\tau$  where $\frac{dt}{d\tau} = g(\V{X}) \vcentcolon= 1/D^2(\V{X})$ and it obeys
\[
d\V{Y}_{\tau} = - \nabla_{\V{Y}}\hat{V}(\V{Y})dt + \sqrt{2kT}d\V{W},
\]
with an effective potential (Theorem \ref{thm:TimeRescalingMD})
\begin{equation}
\label{eqn:NDEffectivePotentialTimeRescaling}
\hat{V}(\V{Y}) = V(\V{RY})- 2kT \ln D(\V{RY}).
\end{equation}
The ergodic theorem generalises to (Theorem \ref{thm:TimeRescalingMDErgodic})
\begin{equation}
\label{eqn:NDtimeErgodicTheorem}
\int_{\mathbb{R}^n} f(\V{X}) \rho(\V{X}) d\V{X} = \lim_{\mathcal{T}_{\text{sim}} \rightarrow \infty} \frac{\int_{\tau=0}^{\mathcal{T}_{\text{sim}}} f(\V{R}\V{Y}_\tau) g(\V{R}\V{Y}_\tau) d\tau}{\int_{\tau=0}^{\mathcal{T}_{\text{sim}}} g(\V{R}\V{Y}_\tau) d\tau}.
\end{equation}

\begin{remark*}
The proof of \eqref{eqn:NDtimeErgodicTheorem} does not require the assumption of Brownian dynamics and therefore it holds more generally for SDEs of the form considered in Section \ref{timeRescalingTheory}.
\end{remark*}

\subsection{Combining multivariate transforms}
\label{sec:combiningTransforms}

The Lamperti and time-rescaling transforms can be combined. This allows transforming a wider class of diffusion processes to constant diffusion than is possible when applying either transform separately. Specifically, from the results of Sections \ref{sec:multivariateLamperti} and \ref{sec:multivariateTimeRescaling}, one expects that diffusion matrices of the form $\V{D}(\V{X}) = \V{D}^{(1)}(\V{X})\V{R}\V{D}^{(2)}(\V{X})$, where 
\begin{equation}
\V{D}^{(1)}(\V{X}) = \begin{bmatrix}
D(\V{X}) & & \\
& \ddots & \\
& & D(\V{X})
\end{bmatrix}, \quad 
\V{D}^{(2)}(\V{X}) = \begin{bmatrix}
D_{1}(X_1) & & \\
& \ddots & \\
& & D_{n}(X_n)
\end{bmatrix},
\label{eqn:Dforms}
\end{equation}
can be transformed to constant diffusion. However, in this case, the transformed process has a non-conservative drift force unless $\V{R}$ is proportional to the identity (Theorem \ref{thm:NoDRDDiffusion}). In particular, if $R_{ij} = \delta_{ij}$, then the process can be transformed to a constant-diffusion Brownian dynamics process $\V{Y}_{\tau}$ through a time rescaling followed by a Lamperti transform, i.e.
\[
\V{X}_t \xrightarrow{\frac{dt}{d\tau}=g(\V{X}) \vcentcolon= D(\V{X})^{-2}} \V{X}_{\tau} \xrightarrow{\V{Y}_\tau = \int^{Y_{i,\tau}} D_i(x)^{-1}dx} \V{Y}_\tau.
\]
The effective potential of the transformed process is (Theorem \ref{thm:DDDiffusion})
\[
\hat{V}(\V{Y}) = V(\V{Y}) - 2kT \ln{D(\phi^{-1}(\V{Y}))} - kT \sum_{i=1}^n \ln D_i(\phi^{-1}_i(\V{Y})),
\]
and the ergodic theorem is
\[
\int_{\mathbb{R}^n} f(\V{X})\rho(\V{X})d\V{X} = \lim_{\mathcal{T}_{\text{sim}} \rightarrow \infty} \frac{\int_{0}^{\mathcal{T}_{\text{sim}}} f(\phi^{-1}(\V{Y}_\tau))g(\phi^{-1}(\V{Y}_{\tau}))d\tau}{\int_{0}^{\mathcal{T}_{\text{sim}}} g(\phi^{-1}(\V{Y}_\tau))d\tau}.
\]

\section{Numerical experiments in one dimension}
\label{sec:NumericalExp1D}



\begin{figure}
    \centering
    \includegraphics{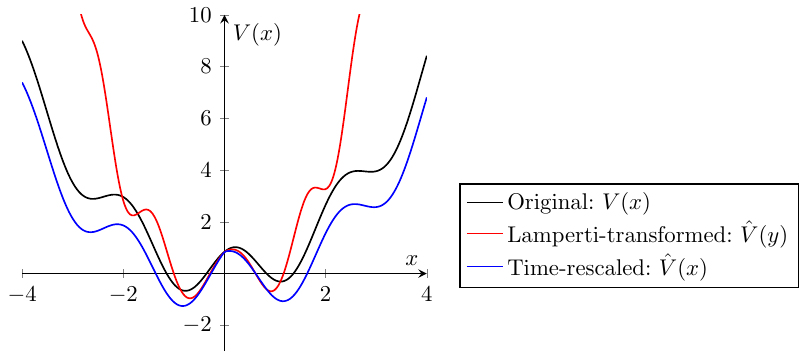}
    \caption{Comparison of effective potentials with original diffusion $D(x) = 1 + \vert x \vert$, $V(x) = \frac{x^2}{2} + \sin(1+3x)$ and $kT=1$. The original potential is in black, the Lamperti-transformed potential is in red and the time-rescaled potential is in blue. Metastability in the potential makes this a challenging sampling problem. We observe that the Lamperti transform stiffens the potential, while the time-rescaling softens it.}
    \label{fig:comparingLampertiandTime}
\end{figure}

 In this section, we simulate Brownian dynamics trajectories for a particular one-dimensional potential with variable noise. Specifically, we consider the potential $V(x)=\frac{x^2}{2} + \sin(1+3x)$ with diffusion profile $D(x) = 1+|x|$ and $kT=1$. Applying the results of Section \ref{sec:OneDimensionTransforms}, the Lamperti-transformed effective potential is
 \[
\hat{V}(y) = V\left(\frac{y}{4}(\vert y \vert +4)\right) - \frac{kT}{2} \ln \left\vert 1 + \vert y \vert + \frac{y^2}{4} \right\vert
\]
and the time-rescaled effective potential is
\[
\hat{V}(x) = V(x) - kT \ln{(1+\vert x \vert)}.
\]
\noindent
In Figure \ref{fig:comparingLampertiandTime}, we sketch $V(x)$, $\hat{V}(y)$ and $\hat{V}(x)$ for $kT=1$. 



For this setup, we consider various numerical integrators (introduced below) both with and without the application of transforms. We compare the weak convergence to the invariant distribution, the sampling efficiency, and the effect of transforms on estimates of the autocorrelation function and the evolving distribution. All experiments are run on a Thinkpad P17 with a 12-core, 2.60GHz Intel i7-10750H CPU, using code implemented in Julia 1.8.5\footnote{GitHub repository: \href{https://github.com/dominicp6/Transforms-For-Brownian-Dynamics}{https://github.com/dominicp6/Transforms-For-Brownian-Dynamics}}.





\subsection{Numerical integrators}
\label{sec:numericalIntegrators}


 

 

We consider the following numerical integrators: Euler-Maruyama (EM), Milstein Method (MM), Leimkuhler-Matthews (LM), Hummer-Leimkuhler-Matthews (HLM), Stochastic Heun (SH), and Limit Method with Variable Diffusion (LMVD). The defining equations of these integrators in the context of one-dimensional Brownian dynamics can be found in Appendix \ref{appdx:Integrators}. The integrators can be summarised as follows:

The Euler-Maruyama (EM) integrator extends the Euler method to SDEs. It has a strong convergence order of $1/2$ and a weak convergence order of $1$ \cite{hutzenthaler_strong_2011}. The Milstein Method (MM) modifies EM by incorporating a second-order correction derived from a stochastic Taylor series expansion. It is strong order $1$ and weak order $1$ 
 and reduces to EM for constant diffusion \cite{milshtejn_approximate_1975}. The Leimkuhler-Matthews (LM) integrator is derived from the high-friction limit of the BAOAB-splitting method in the constant diffusion regime \cite{leimkuhler_rational_2013}. It has weak convergence order $2$ for constant diffusion but is invalid (does not converge) for multiplicative noise. The Hummer-Leimkuhler-Matthews (HLM) integrator is an extension of LM that ensures that the expectation of position is exact in the case of locally linear diffusion, and is conjectured to improve convergence in the variable diffusion regime\footnote{We would like to thank Gerhard Hummer for suggesting this method in personal correspondence.}. It reduces to LM for constant diffusion. The Stochastic Heun (SH) integrator is a two-stage Runge-Kutta method. It has weak convergence order of $2$ for constant diffusion, otherwise order $1$ for variable diffusion \cite{blum_numerical_1972}. However, the accuracy gains of SH come at the cost of higher computational requirements, as it involves two force evaluations, two diffusion coefficient evaluations, and two diffusion gradient evaluations per iteration. The Limit Method with Variable Diffusion (LMVD) is a scheme that has a weak convergence order of $2$ for both constant and variable diffusion. It stems from the high-friction limit of the BAOAB-splitting method in the variable diffusion regime. Unlike SH, it requires one force evaluation per iteration, however, it requires two ODE solves per timestep, making it comparatively expensive compared to the related LM method. It is a novel integrator method that we introduce in this work. The derivation can be found in Appendix \ref{appdx:LMVD}. Note that LMVD reduces to LM for constant diffusion regime.

\subsection{Error in infinite time}
\label{sec:errorInInfiniteTime}

\begin{figure}
  \centering
  \begin{subfigure}{0.49\textwidth}
    \centering
    \includegraphics[width=\linewidth]{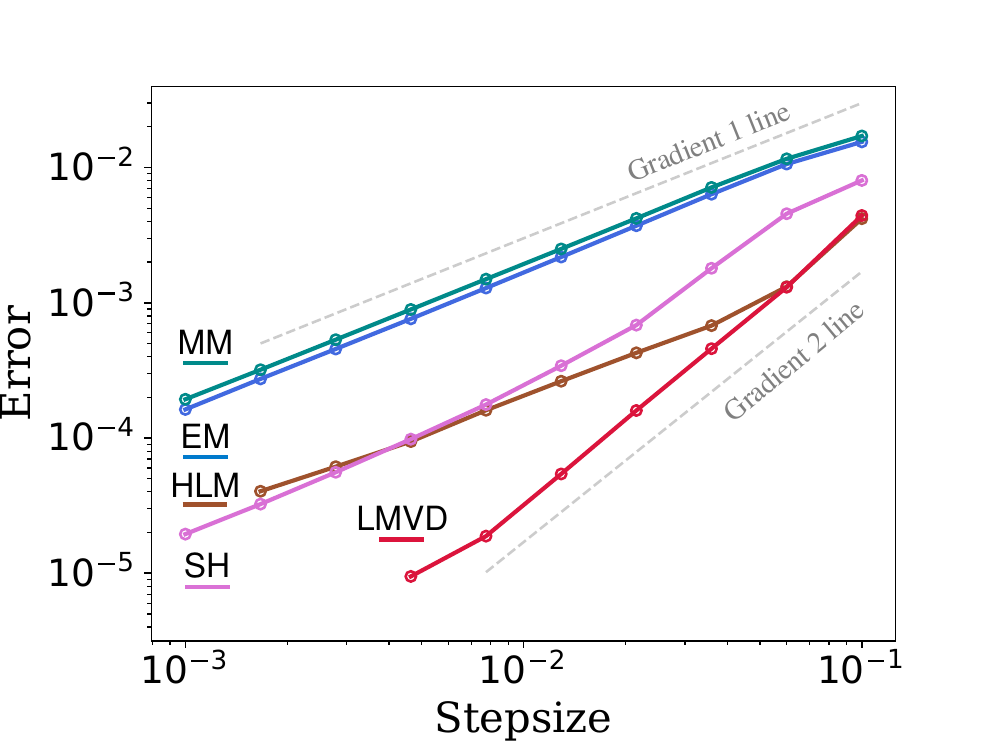}
    \caption{}
  \end{subfigure}
  \hfill
  \begin{subfigure}{0.49\textwidth}
    \centering
\includegraphics[width=\linewidth]{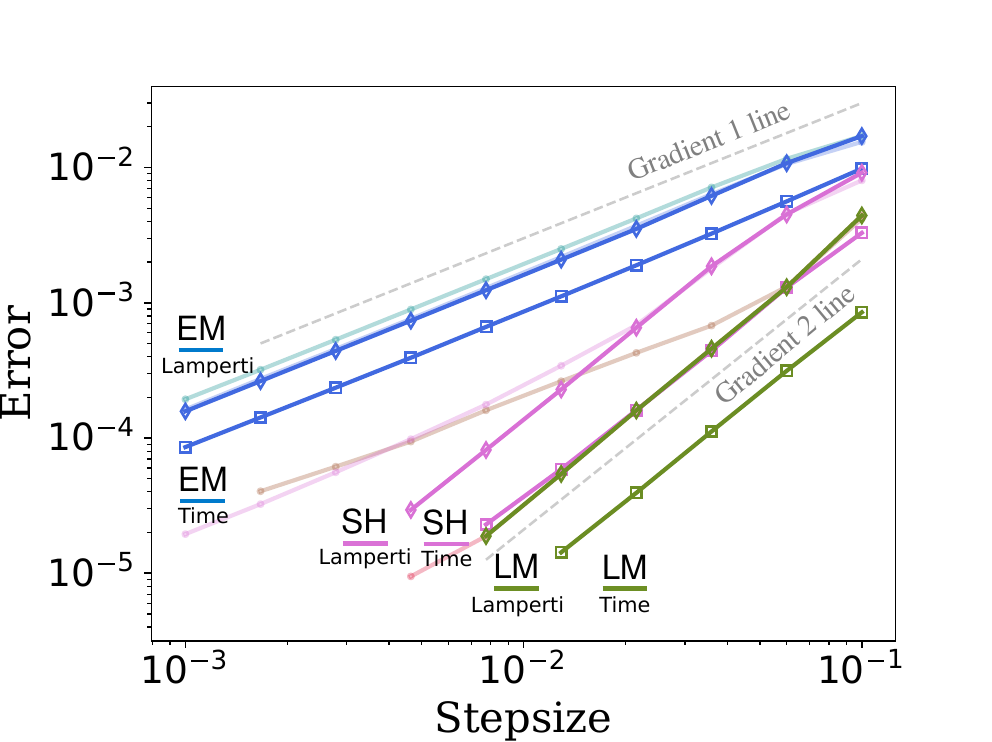}
    \caption{}
  \end{subfigure}
  \caption{Rates of convergence to the invariant measure. The simulation time was fixed at $T_{\text{sim}}, \mathcal{T}_{\text{sim}} = 7.5 \times 10^7$ and $12$ independent runs were averaged to further reduce sampling errors. (a) The untransformed methods. (b) When applying a transform to constant diffusion, either a Lamperti transform or a time rescaling. The untransformed methods are shown in faint in panel (b) to facilitate comparison. 
  }
  \label{fig:1DConvergence}
\end{figure}

We compare weak convergence to the invariant distribution $\rho(x) \propto \exp{(-V(x)/kT)}$ with varying step size $h$, using trajectories generated by the different integrators both with and without transforms to constant diffusion. For untransformed dynamics, we compare EM, MM, HLM, SH, and LMVD. For the Lamperti-transformed dynamics and the time-rescaled dynamics, we compare the EM, LM, and SH integrators. We omit MM since it reduces to EM for constant diffusion, while LMVD and HLM both reduce to LM for constant diffusion. For each method, we run trajectories of length $T_{\text{sim}}, \mathcal{T}_{\text{sim}} = 7.5 \times 10^7$, and 12 independent runs are averaged to reduce sampling errors.

To assess the convergence of the invariant distribution, we divide a subset of the $x$ domain into $M$ equal-length intervals and compute the mean error between the empirical probabilities and the exact probabilities given by the invariant distribution. For Lamperti-transformed experiments, we derive empirical probabilities using equation \eqref{reweightingLamperti}, for time-rescaled experiments we use equation \eqref{reweightingTimeRescaling}. We compute the L1 error:
\begin{equation}
\label{eqn:invariantDistributionError}
    \text{Error} \vcentcolon= \frac{1}{M} \sum_{i=1}^M \vert \omega_i - \hat{\omega}_i \vert, 
\end{equation}
where $\omega_i$ is the exact occupancy probability of the $i^{th}$ interval and $\hat{\omega}_i$ is the empirical estimate. We use $30$ equal-width intervals in the range $-5$ to $5$ and run each integrator using $10$ different step sizes, equally spaced in log space between $10^{-3}$ and $10^{-1}$. Steps are in $\tau$-time for time-rescaled methods, but $t$-time for all other methods. The error is plotted against the step size on a log-log scale, so first-order weak methods have a gradient of one, and second-order weak methods have a gradient of two. The results are shown in Figure \ref{fig:1DConvergence}. 


Figure \ref{fig:1DConvergence}(a) confirms the expected orders of weak convergence for the untransformed methods. Notably, MM has a larger error constant than EM, illustrating that improved strong convergence doesn't guarantee improved weak convergence. Examining Figure \ref{fig:1DConvergence}(b), we see that the effect of applying a transform is also integrator-dependent. Applying a transform to EM hardly changes the convergence properties, while applying a transform to constant diffusion for SH or LM restores their second-order convergence behavior, with the transformed LM method closely following the convergence properties of LMVD.  Note that the time-transformed methods are shifted relative to the Lamperti-transformed methods. This is expected because of the discrepancy between a constant step size in $\tau$-time and $t$-time. It does not automatically imply that time-rescaled integrators result in more efficient samplers than Lamperti-transformed integrators. We investigate the question of efficiency in Section \ref{sec:ComputationalEfficiency} below.

\begin{figure}
  \centering
  \begin{subfigure}{0.49\textwidth}
    \centering
    \includegraphics[width=\linewidth]{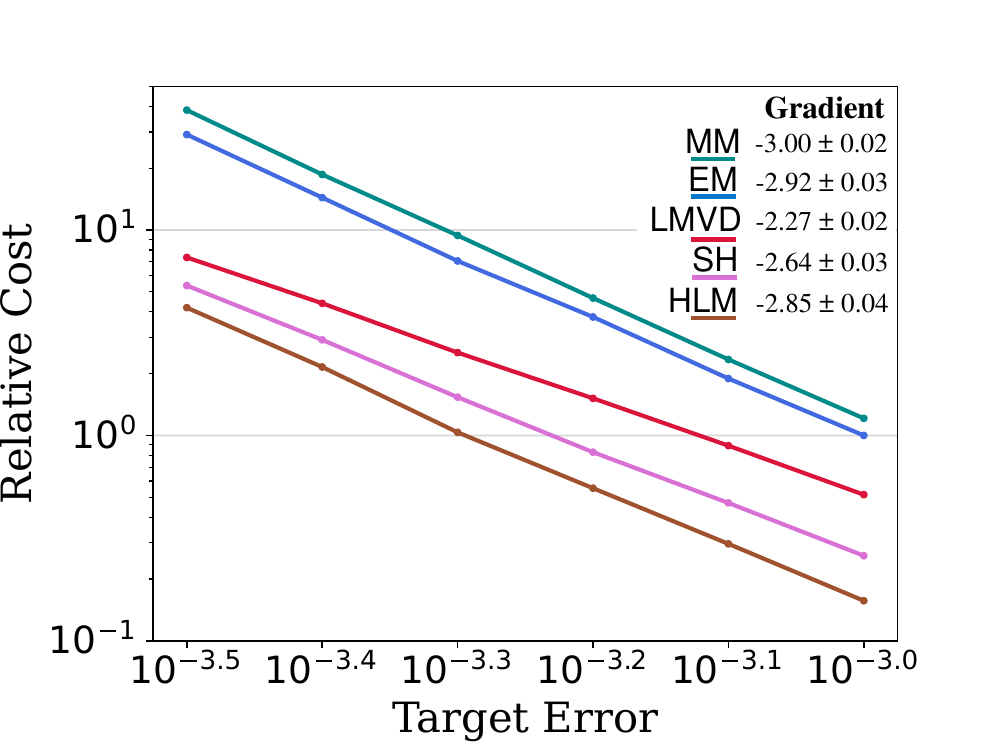}
    \caption{}
  \end{subfigure}
  \hfill
  \begin{subfigure}{0.49\textwidth}
    \centering
\includegraphics[width=\linewidth]{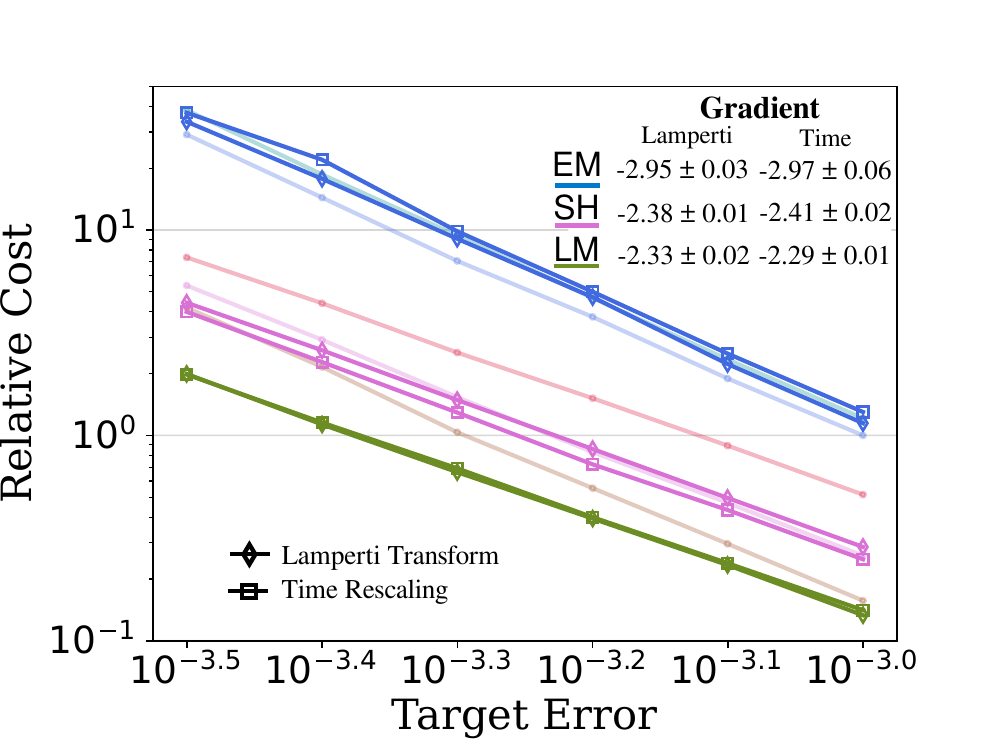}
    \caption{}
  \end{subfigure}
  \caption{Cost-error diagram to compare numerical efficiency. Error is defined as per Equation \eqref{eqn:invariantDistributionError}. The cost is the wall-clock time to reach the target error (number of iterations times cost per iteration), relative to the wall-clock time for untransformed EM to reach a target error of $10^{-3}$. (a) Untransformed methods. (b) When applying constant-diffusion transform: Lamperti or time rescaling. The untransformed methods are shown in faint for comparison.}
\label{fig:computationalEfficiency}
\end{figure} 

\subsection{Computational efficiency and numerical stability}
\label{sec:ComputationalEfficiency}

\begin{table}[h!]
\centering
\begin{tblr}{
  colspec = {X[c]*{6}{X[c]}},
  hline{2-3} = {0.4pt}, 
}
\hline[0.4pt] 
& \SetCell[c=2]{c} {\textbf{Untransformed}} && \SetCell[c=2]{c} {\textbf{Lamperti}} && \SetCell[c=2]{c} {\textbf{Time-rescaling}} \\
Integrator & $t \, (s)$ & $h^*$ & $t \, (s)$ & $h^*$ & $t \, (s)$ & $h^*$ \\
 EM & $10.83(8)$ & $0.20$ & $12.77(8)$ & $0.25$ & $13.18(9)$ & $0.25$ \\
SH & $14.47(13)$ & $0.25$ & $17.22(8)$ & $0.32$ & $15.56(14)$ & $2.5$\\
LM & - & - & $12.54(8)$ & $0.25$ & $13.09(4)$ & $2.5$ \\
MM & $11.97(18)$ & $0.16$ & - & - & - & -\\
HLM & $10.72(10)$ & $0.20$ & - & - &  - & -\\
LMVD & $48.20(12)$ & $0.25$ & - & - & - & -\\
\hline
\end{tblr}
\caption{Stability thresholds (\(h^*\)) and the compute time required for $10^8$ iterations (\(t\)) compared for untransformed and transformed methods. Standard errors in $t$ were computed by averaging 12 runs with a constant step size of 0.01. Stability thresholds were determined as the first step size (in geometric increments of \(10^{0.1}\)) that resulted in numerical blow-up. Errors are in bracket notation, e.g., \(10.83(8) = 10.83 \pm 0.08\).}
\label{table:computationalCost2}
\end{table}

We assess the computational efficiency of each method by comparing the wall-clock time required to achieve a fixed L1 error in the invariant measure, as defined by Equation \eqref{eqn:invariantDistributionError}. We estimate the wall-clock cost per iteration of each method by timing $10^8$ iterations with a fixed step size of $h=0.01$, averaging over 12 runs. For transformed methods, any additional cost of applying the counting formulas \eqref{reweightingLamperti} or \eqref{reweightingTimeRescaling} is included in these timings. We then fix a target error and run trajectories with various step sizes for each method, stopping when first reaching the target error. For each step size, we average 6,000 repeats and find the minimum number of iterations to reach the specified error. The total wall-clock time is then estimated as the minimum number of iterations over the various step sizes times the cost per iteration. This calculation is repeated for 5 target errors logarithmically spaced between $10^{-3.5}$ and $10^{-3}$. The resulting cost-error diagram is illustrated in Figure \ref{fig:computationalEfficiency}. 

Numerical stability is estimated by determining the smallest step size, in logarithmic increments of $10^{0.1}$, where numerical blow-up occurs before $T_{\text{sim}} , \mathcal{T}_{\text{sim}} =10^6$. These stability threshold estimates as well as timing results for $10^8$ iterations are shown in Table \ref{table:computationalCost2}. We see that time-rescaling significantly improves the stability threshold of SH and LM in this case whereas the Lamperti transform does not. This is explained by the fact that, for this diffusion coefficient, the time-rescaled potential is the softer of the two transformed potentials (Figure \ref{fig:comparingLampertiandTime}). Thus, for this problem,  the time-rescaling approach might be preferable if simulations with large step sizes are required.

As for the efficiency of the untransformed methods (Figure \ref{fig:computationalEfficiency}(a)), we note that the method with the best weak convergence (LMVD) isn't necessarily the most computationally efficient for a given range of target errors (HLM). Furthermore, Figure \ref{fig:computationalEfficiency}(b) shows how applying coordinate transforms can significantly improve the computational efficiency of certain integrators but have a more modest impact on others. For example, the Lamperti-transformed/time-rescaled LM is the most computationally efficient method overall - approximately 5 times more efficient than LMVD. However, applying transforms only slightly improves the efficiency of SH, and even slightly reduces the efficiency of EM. In general, both types of transform have very similar effects on efficiency, and any differences can be attributed to small differences in the iteration cost (Table \ref{table:computationalCost2}). Overall, we see that transformations only improve the efficiency of numerical integrators that have better convergence properties for constant noise. 

\subsection{Error in finite time}

\begin{figure}
    \centering
    \begin{subfigure}{0.48\textwidth}  
        \includegraphics[scale=0.475]{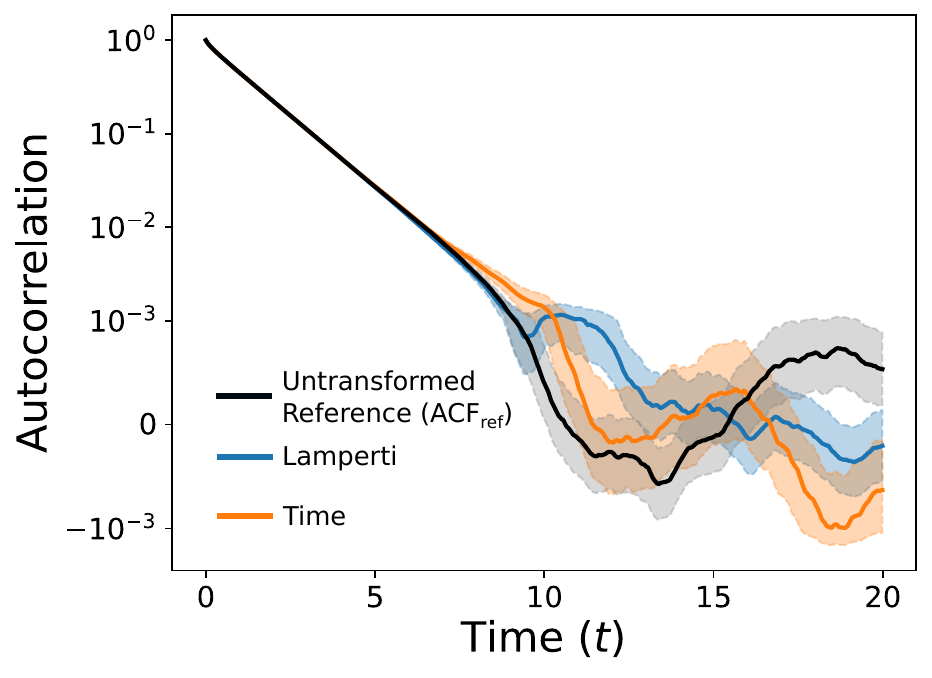}
        \caption{}
        \label{fig:subfigure1}
    \end{subfigure}
    \hfill  
    \begin{subfigure}{0.48\textwidth}  
        \includegraphics[scale=0.475]{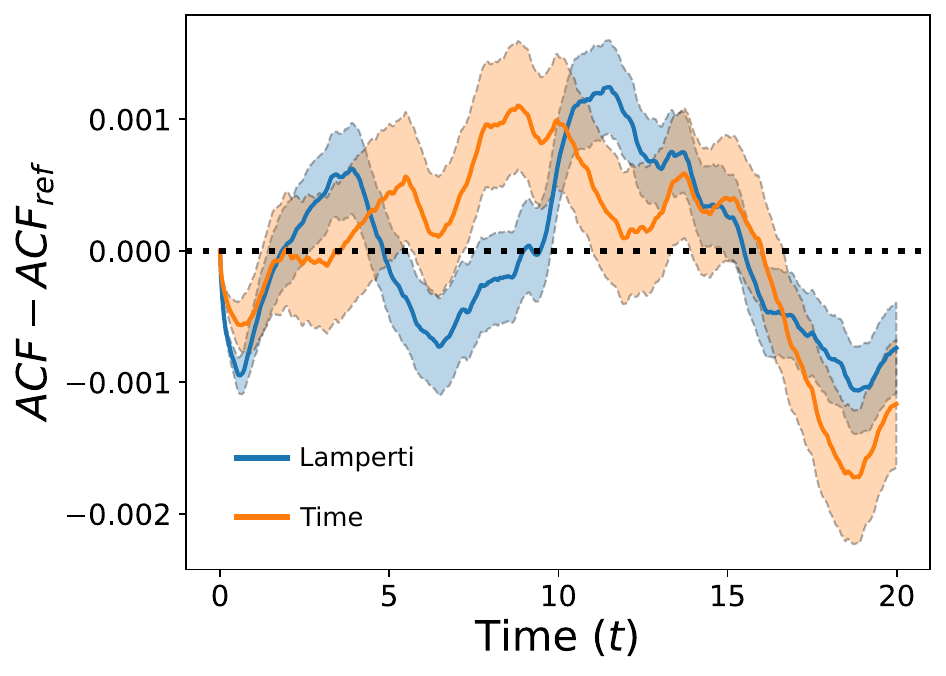}
        \caption{}
        \label{fig:subfigure2}
    \end{subfigure}
    \caption{Comparing normalised autocorrelation function estimates with and without transforms. Panel (a) shows the mean and standard deviation in the mean of autocorrelation function estimates obtained using the Stochastic Heun integrator with 200 trajectories of length $T_{\text{sim}}, \mathcal{T}_{\text{sim}}=5000$ and step size $h=0.01$. In black is the reference (no transform applied), in blue is when using a Lamperti transform and in orange is when using time rescaling. Panel (b) displays mean and standard error of the differences in autocorrelation function estimates using the Lamperti transform (blue) and time-rescaling transform (orange) compared to the reference ($ACF_{ref}$). Overall, biases introduced by the transforms are small and only practically significant at large times. }
    \label{fig:autocorrelation}
\end{figure}

\begin{figure}
  \centering
  \begin{subfigure}{0.49\textwidth}
    \centering
    \includegraphics[width=\linewidth]{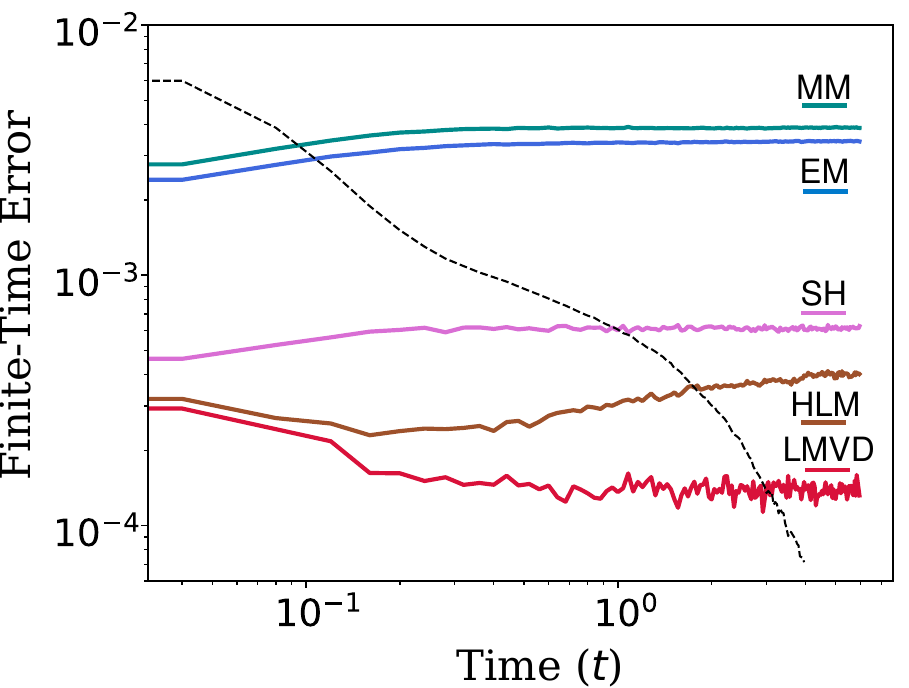}
    \caption{}
  \end{subfigure}
  \hfill
  \begin{subfigure}{0.49\textwidth}
    \centering
\includegraphics[width=\linewidth]{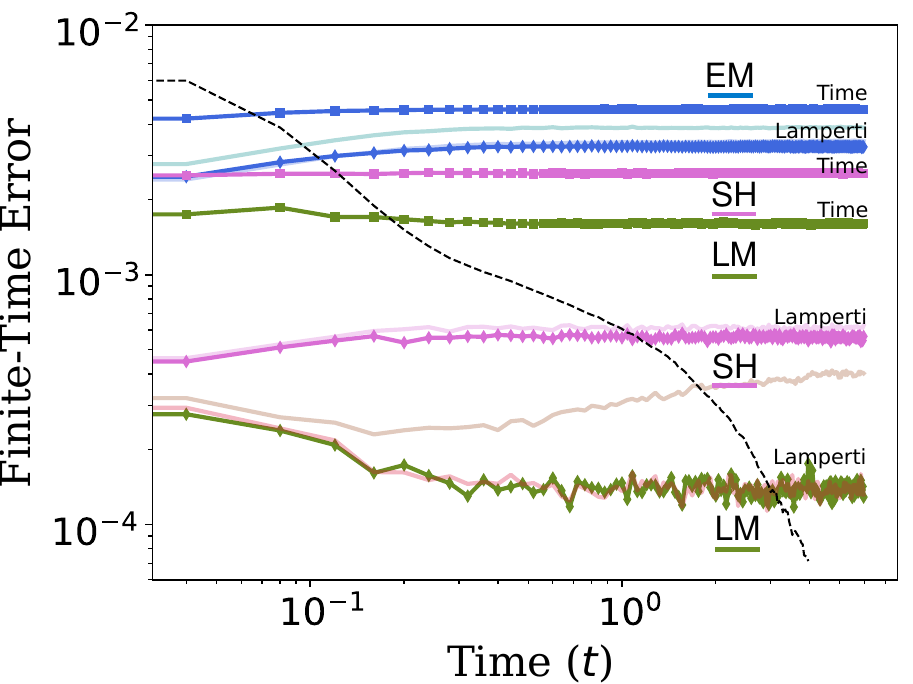}
    \caption{}
  \end{subfigure}
  \caption{Finite-time errors of the evolving distribution in the interval $t \in [0, 6]$ for fixed step size $h=0.02$. The reference distribution at time $t$ is computed by averaging over $2.5 \times 10^7$ independent trajectories using the SH method with small step size $h=10^{-4}$. In each plot, the dotted black line represents the L1 difference between the reference evolving distribution and the invariant distribution. (a) Errors of the untransformed methods. (b) Errors when applying a transform to constant diffusion, either a Lamperti transform or a time rescaling. The untranformed methods are shown in faint to facilitate comparison. Integrators are Leimkuhler-Matthews (LM), Milstein Method (MM), Euler-Maruyama (EM), Hummer-Leimkuhler-Matthews (HLM), Stochastic Heun (SH), Limit Method with Variable Diffusion (LMVD).}
   \label{fig:finiteTimeErrorEvolution}
\end{figure}


In the continuous limit, the Lamperti and time-rescaling transforms are invertible, allowing recovery of the original dynamics. However, the transforms of discredited dynamics could introduce additional bias for numerics. This section explores the effect of transforms on estimates of dynamic quantities, specifically the autocorrelation function and the evolving distribution. 

\subsubsection{Autocorrelation function}

To obtain a reference estimate of the autocorrelation function, we run $200$ randomised trajectories of length $T=5000$ using the Stochastic Heun integrator with step size of $h=0.01$, and with initial positions drawn from a standard normal distribution. For each trajectory, we use the Fast Fourier Transform (FFT) algorithm to estimate the normalised autocorrelation function and compute the mean and standard deviation of the best estimate. Additionally, we run trajectories under the same parameters but separately apply a Lamperti transform and time-rescaling transform. We then transform these trajectories back to $x$-space and $t$-time respectively and compute the normalised autocorrelation function. The three autocorrelation functions so obtained (reference, Lamperti and time-rescaling) are shown in Figure \ref{fig:autocorrelation}(a). We also calculate the mean and standard deviation in the mean of the difference between the Lamperti/time-rescaling autocorrelation functions and the reference estimate. These results are displayed in Figure \ref{fig:autocorrelation}(b).

Overall, at short times (approximately $t \lesssim 8$) the differences in autocorrelation estimates are generally minimal and often not statistically significant. As time increases (between $10$ and $20$), the fractional error in the mean becomes more significant, but still much smaller than the standard deviation across runs (which is $\sqrt{200} \approx 14$ times larger than the standard error in this case). 

Note that while inverting the Lamperti transform is straightforward (apply the inverse coordinate transformation), undoing the time-rescaling process can be more challenging. This is because direct conversion from $\tau$-time to $t$-time results in an irregular time series, making it unsuitable for direct application of FFT-based methods. To address this, we perform linear interpolation on a $t$ grid with the same regular spacing of $h=0.01$ before applying the FFT. This interpolation step introduces bias. However, as we have seen in Figure \ref{fig:autocorrelation}, the overall bias remains small and unlikely to be practically significant. Alternatively, methods designed for unevenly spaced time series, such as least-squares spectral analysis, could be used but these add significant computational cost, negating any sampling efficiency benefits of the time transform. 

\subsubsection{Evolving distribution}

With initial positions drawn from a standard normal distribution, we estimate the evolving distribution using $2.5 \times 10^7$ independent trajectories made with the SH integrator, step size $h=10^{-4}$ and compare this to the evolving distribution estimates computed with step size $h=0.02$ for each method.  We compute the L1 errors with respect to the reference distribution at time snapshot intervals of $\delta t = 0.04$, using the same histogram binning as introduced in Section \ref{sec:errorInInfiniteTime}. 

For methods involving the Lamperti transform, the initial condition is transformed to $y$-space, trajectories are evaluated and then transformed back to $x$-space. For methods involving time rescaling, trajectories are evaluated in $\tau$-time and then transformed back to $t$-time. However, the conversion from $\tau$-time to $t$-time is problematic since this transform depends on the unique history of each trajectory. To overcome this, we choose to simulate each trajectory in $\tau$-time (in steps of $h=0.02$) until slightly overshooting the $\delta t = 0.04$ interval. The position at the required $t$-time is then estimated by linear interpolation with the previous sample. This approach is inexpensive but can introduce bias.

The resulting errors for the untransformed and transformed methods are shown in Figure \ref{fig:finiteTimeErrorEvolution}(a) and \ref{fig:finiteTimeErrorEvolution}(b), respectively. Additionally, the figures include a dotted black line representing the L1 difference between the reference evolving distribution and the invariant distribution. For points below this line, the evolving distribution is distinguishable from the invariant distribution.

Examining the untransformed methods in Figure \ref{fig:finiteTimeErrorEvolution}(a), we observe that by $t=6$, the errors for each method have already converged to their corresponding infinite-time errors depicted in Figure \ref{fig:1DConvergence} (a), which is consistent with the correlation timescale implied in Figure \ref{fig:autocorrelation}. However, we see that MM, EM and LM are unsuitable at this step size if high accuracy is required, as their errors soon exceed the difference between the evolving and invariant distributions.

Examining the transformed methods, we see that the Lamperti transform has no detrimental impact on finite-time errors. In particular, the Lamperti-transformed LM method has the same finite-time error as the more expensive LMVD method. By contrast, the time-rescaled methods show a noticeable bias, likely originating from the need for linear interpolation when computing the evolving distribution at fixed $t$. This bias can make these methods ill-suited for high-accuracy simulations of the evolving distribution in practice. 


\section{Multivariate numerical experiments}
\label{sec:MultivariateExperiments}

\begin{figure}
    \centering
    \includegraphics[scale=0.4]{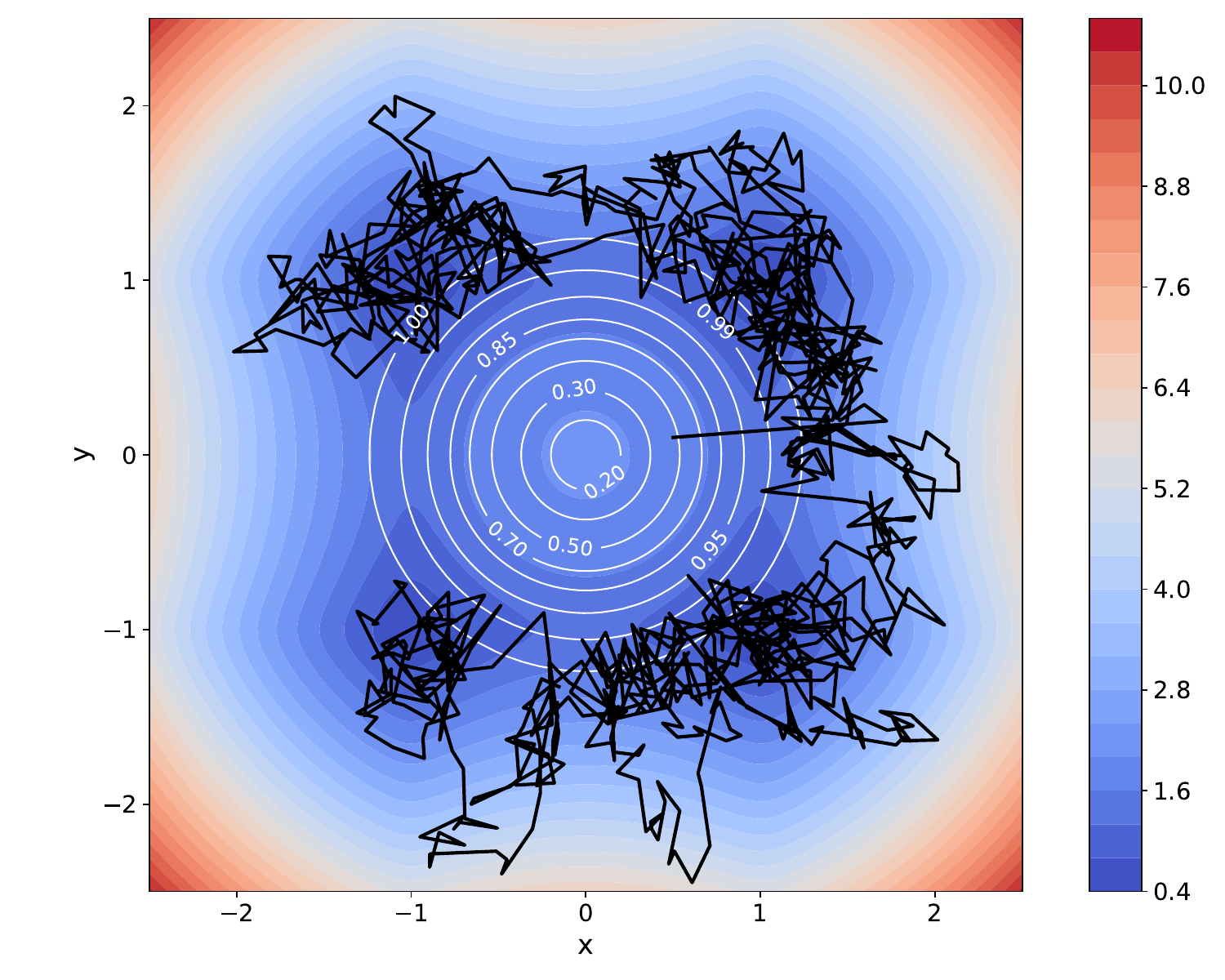}
    \caption{Heatmap of the quadruple-well potential function \eqref{quadrupleWellPotential}. White circles depict contours of the Frobenius norm of the Moro-Cadin diffusion tensor \eqref{moroCadinDiffusion}. The black path shows a Euler-Maruyama trajectory of Brownian dynamics of $1000$ steps with time step of $h=0.01$ and $kT=1$. Note the small norm of the diffusion tensor in the vicinity of the origin. This inhibits hopping between the wells, making this a more challenging sampling problem.}
    \label{fig:morocadin}
\end{figure}

\begin{figure}
  \centering
  \begin{subfigure}{0.49\textwidth}
    \centering
    \includegraphics[width=\linewidth]{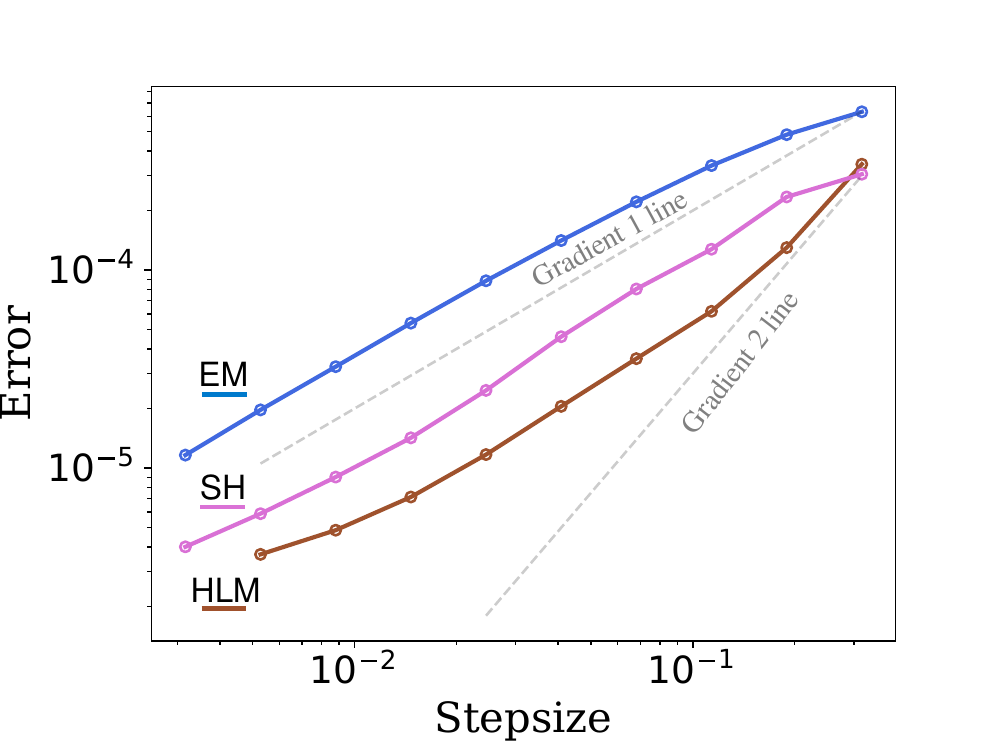}
    \caption{}
  \end{subfigure}
  \hfill
  \begin{subfigure}{0.49\textwidth}
    \centering
\includegraphics[width=\linewidth]{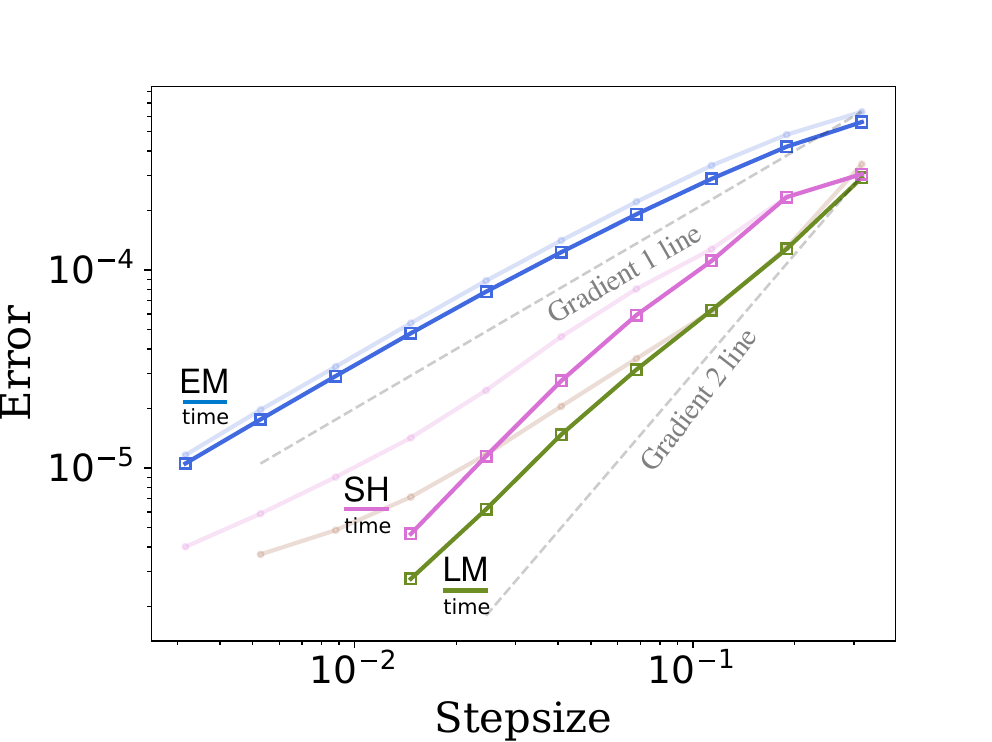}
    \caption{}
  \end{subfigure}
  \caption{Rates of convergence to the invariant measure for Brownian dynamics in a 2D quadruple well potential with Moro-Cadin diffusion tensor. The simulation time was fixed at $T_{\text{sim}},\mathcal{T}_{\text{sim}}=5 \times 10^6$ and $12$ independent runs were averaged to further reduce sampling errors. (a) The untransformed methods. (b) When applying a transform to constant diffusion through a time rescaling. The untransformed methods are shown in faint to facilitate comparison. Integrators are Leimkuhler-Matthews (LM), Euler-Maruyama (EM), Hummer-Leimkuhler-Matthews (HLM) and Stochastic Heun (SH).}
  \label{fig:2DConvergence}
\end{figure}

As an example of multivariate Brownian dynamics, we consider Stokes-Einstein diffusion, which models the diffusion of a low concentration of non-interacting, spherical particles suspended in a fluid. It has widespread applications, particularly in materials science \cite{catlow_chapter_2007, kelton_chapter_2010}, and is also used for modelling water diffusion in MRI imaging applications \cite{alexander_diffusion_2007}. In $n$ dimensions, the Stokes-Einstein diffusion tensor is given by
\[
\V{D}_{SE} = \frac{k_B T}{6 \pi \eta r} \V{1}_n,
\]
where $T$ is Kelvin temperature, $\eta$ is viscosity and $\V{1}_n$ is the $n$-dimensional identity matrix. 

If the temperature field or the fluid's material properties are non-homogeneous, then $\V{D}_{SE}$ is an isotropic, position-dependent matrix, and can be transformed to constant diffusion by time rescaling. Furthermore, if the medium is anisotropic, and certain diffusion directions are preferred over others, then the diffusion model can be generalised to $\V{D}(\V{X})=\V{D}_{SE}(\V{X})\V{D}^{(2)}(\V{X})$, where $\V{D}^{(2)}(\V{X})$ is the diagonal matrix given by equation \eqref{eqn:Dforms}. This kind of diffusion is common in biological tissues, for example in the brain, where positional alignment of white matter tracts results in preferred diffusion directions for water molecules \cite{me_anisotropy_1991}. Importantly, since this diffusion model is of the form $\V{D}(\V{X}) = \V{D}^{(1)}(\V{X})\V{D}^{(2)}(\V{X})$ in equation \eqref{eqn:Dforms}, the process can be transformed to constant-diffusion through the combination of a time-rescaling and a Lamperti transform. 

The example of Stokes-Einstein diffusion we consider is multivariate Brownian dynamics in a 2D quadruple-well potential given by 
\begin{equation}
\label{quadrupleWellPotential}
V(x,y) = \sqrt{\frac{17}{16} -2x^2 + x^4} + \sqrt{\frac{17}{16} - 2y^2 + y^4},
\end{equation}
with a diffusion tensor given by the Moro-Cardin tensor \cite{moro_saddle_1998}
\begin{equation}
\label{moroCadinDiffusion}
D(x, y) = \left(1 + A \exp\left(-\frac{x^2 + y^2}{2 \sigma^2}\right)\right)^{-1}\V{1},
\end{equation}
where $A=5$ and $\sigma=0.3$, see Figure \ref{fig:morocadin}. Since this diffusion tensor is isotropic, it can be mapped to constant diffusion through time rescaling. 

Figure \ref{fig:2DConvergence} illustrates the comparison of weak convergence to the invariant measure for the LM, EM, SH, and HLM integrators. Figure \ref{fig:2DConvergence}(a) is without any transforms, while Figure \ref{fig:2DConvergence}(b) is after a time-rescaling transform to constant diffusion has been applied. We follow the same general approach as first outlined in Section \ref{sec:errorInInfiniteTime}. We run trajectories of length $T_{\text{sim}}, \mathcal{T}_{\text{sim}} = 5 \times 10^6$ and average over $12$ independent runs and we run each integrator using $10$ different step sizes, equally spaced in log-space between $10^{-2.5}$ and $10^{-0.5}$. For histogram computation, we use a $30 \times 30$ grid of equal-width square bins covering the domain $[-3, 3] \times [-3, 3]$ in the $x$-$y$ plane.   

We observe similar behavior to the one-dimensional numerical experiments discussed in Section \ref{sec:errorInInfiniteTime}. It is noteworthy that applying a time-rescaling transform enhances the convergence rate for both the SH and LM integrators. Similar to the one-dimensional case, the transformed LM integrator exhibits a lower error constant compared to SH, indicating its superior efficiency for this particular problem. 

\section{Conclusions}
\label{sec:Conclusions}
We examined two types of transform to constant diffusion for Brownian dynamics with multiplicative noise: the Lamperti transform and the time-rescaling transform. We derived conditions on the noise term for these transforms to be applied and combined. Furthermore, through numerical experiments in one and two dimensions, we have shown how using these transforms, combined with an appropriate SDE integrator, can lead to a highly efficient sampling method for certain classes of multivariate noise. 

For one-dimensional Brownian dynamics, we showed that both transforms are always applicable, regardless of the form of the diffusion coefficient. However, the two transforms affect the dynamics differently, so the choice of transform may depend on the specific problem at hand. We showed numerically that applying either transform with the Leimkuhler-Matthews (LM) integrator significantly improves the convergence to the invariant measure, resulting in a method that has approximately five times higher sampling efficiency than the Limit Method with Variable Diffusion (LMVD) - a highly-performant integrator for multiplicative noise that does not utilise transformations. This transformed method also significantly outperformed the popular Euler-Maruyama integrator, with a 10 to 25 times higher computational efficiency for the problem investigated. Crucially, this method only requires one force and one diffusion tensor evaluation per iteration, thus scaling better to high-dimensional problems than competing methods that require multiple force and/or diffusion evaluations per step. 

In addition to investigating convergence to the invariant measure, we also verified whether dynamics information, in the form of the autocorrelation function and the evolving distribution, can be recovered after simulating a transformed process and then applying the inverse transform. We found that the Lamperti transform introduced no appreciable bias for estimates of either quantity, but that the time-rescaling transform is less suitable for recovering finite-time distributions.

For multivariate Brownian dynamics, the Lamperti and time-rescaling transforms have somewhat limited application. However, the two transformations can be combined to transform non-homogeneous, anisotropic Stokes-Einstein diffusion into a constant diffusion process. This is a broad class of diffusion tensors with applications in biological diffusion processes. Furthermore, in the context of rare event sampling, the flexibility of the coordinate transforms offers an additional advantage. In Section \ref{sec:EnhancedSampling} we showed that they can be used to design diffusion profiles that, when transformed back to constant diffusion, lead to enhanced sampling. In ongoing research, we are actively exploring these techniques in diverse applications and are investigating further refinements to optimise their effectiveness in enhancing sampling methodologies. We anticipate that coordinate transforms will improve the efficiency of Brownian dynamics simulations in various contexts.

\section{Acknowledgements}

This work was supported by the United Kingdom Research and Innovation (grant EP/S02431X/1), UKRI Centre for Doctoral Training in Biomedical AI at the University of Edinburgh, School of Informatics.

We extend our gratitude to the reviewers for their insightful comments and constructive suggestions, which greatly improved the quality and clarity of this manuscript.

\section{Data availability statement}
The software and data that support the findings of this study are openly available in Zenodo at https://zenodo.org/badge/latestdoi/616585156.

\newpage

\bibliographystyle{tfo}
\bibliography{bibliography}

\newpage

\appendix
\section{Double-well enhanced sampling}
\label{appx:doubleWellEnhancedSampling}

\noindent
\textbf{Transition counts in a double well}. Given a Brownian dynamics trajectory $X = (x_0, x_1, \dots, x_n)$, we define the transition counts $N_T$ to be the total number of times that the trajectory's state passes directly from the left inflection point of the double well to the right inflection point (or vice versa) through a sequence of intermediate states. 

\vspace{0.2cm}
\noindent
\textbf{$L_1$ error in the invariant measure}. We use an identical method to the method of computation of $L_1$ discussed in Section \ref{sec:errorInInfiniteTime}. We use $30$ equal-width intervals in the range $-5$ to $5$.

\vspace{0.2cm}
\noindent
\textbf{Methodology details for Figure} \ref{fig:L1_error_compared}. For each $kT$ value and integrator, experiments were run with $30$ values for $\alpha$ equally spaced between $0$ and $1$ inclusive. For each $\alpha$ value, $5 \times 10^5$ independent repeats were run for the simulations of length $T_{\text{sim}}, \mathcal{T}_{\text{sim}}=10$ and $5 \times 10^3$ repeats for the simulations of length $T_{\text{sim}}, \mathcal{T}_{\text{sim}}=1000$ and the mean $L_1$ error  was calculated. The minimum $L_1$ error across all considered $\alpha$ values is plotted in Figure \ref{fig:L1_error_compared}.

\section{Numerical Integrators}
\label{appdx:Integrators}

We use the shorthand notation
\[
    \begin{split}
        a(x) &\vcentcolon= -D(x)\frac{dV}{dx} + kT\frac{dD}{dx}, \\
        \sigma(x) &\vcentcolon= \sqrt{2kTD(x)}, \\
        \tilde{a}(x) &\vcentcolon= a(x) - \frac{1}{2}\sigma(x)\frac{d\sigma}{dx} = -D(x)\frac{dV}{dx} + \frac{1}{2}kT \frac{dD}{dx},
    \end{split}
\]
where $a(x)$ is the drift term, $\sigma(x)$ the diffusion term, and $\tilde{a}(x)$ is the Stratonovich-corrected drift \cite{pavliotis_stochastic_2014}. We consider the following integrators, where $w_n, w_{n+1} \overset{\mathrm{iid}}{\sim} \mathcal{N}(0,1)$ and $h$ is the step size:

\begin{itemize}
    \item[i)] Euler-Maruyama (EM)
\[
x_{n+1} = x_n + a(x_n)h + \sigma(x_n)\sqrt{h}w_n;
\]
\item[ii)] Milstein Method (MM)
\[
x_{n+1} = x_n + a(x_n)h + \sigma(x_n)\sqrt{h}w_n + \frac{1}{2}kT \frac{dD}{dx}(x_n)(w_n^2 - 1)h;
\]
\item[iii)] Leimkuhler-Matthews (LM)
\[
x_{n+1} = x_n + a(x_n)h + \sigma(x_n)\sqrt{h}\frac{w_n + w_{n+1}}{2};
\]
\item[iv)] Hummer-Leimkuhler-Matthews (HLM)
\[
\label{HLM}
x_{n+1} = x_n + \left(a(x_n) + \frac{1}{4}kT\frac{dD}{dx}(x_n)\right)h + \sigma(x_n)\sqrt{h}\frac{w_n + w_{n+1}}{2};
\]
\item[v)] Stochastic Heun (SH)
\[
    \label{SH}
    \begin{split}
        x^*_{n+1} &= x_n + \tilde{a}(x_n)h + \sigma(x_n)\sqrt{h}w_n \\
        x_{n+1} &= x_n + \frac{1}{2}\left(\tilde{a}(x_n) + \tilde{a}(x^*_{n+1})\right)h + \frac{1}{2}\left({\sigma}(x_n) + {\sigma}(x^*_{n+1})\right)\sqrt{h}w_n;
    \end{split}
\]
\item[vi)] Limit Method with Variable Diffusion (LMVD)
\begin{equation} \label{eqn:LMVD_defn}
    \begin{split}
        \hat{x}_{n+1} = \sqrt{kT}w_n - \sqrt{2h D(x_n)} \frac{dV}{dx}(x_n) + kT\sqrt{\frac{h}{2 D(x_n)}}\frac{dD}{dx}(x_n) \\
        \tilde{x}_{n+1} = \left\{x\left(\sqrt{\frac{h}{2}}\right) \bigg| \; x(0) = x_n, \; dx = \sqrt{D(x)}\hat{x}_{n+1}dt \right\} \\
        x_{n+1} = \left\{x\left(\sqrt{\frac{h}{2}}\right) \bigg| \; x(0)=\tilde{x}_{n+1}, \; dx=\sqrt{kT D(x)}w_{n+1}dt\right\}.
    \end{split}
\end{equation}


\end{itemize}
\noindent

\section{Derivation of the Limit Method with Variable Diffusion}
\label{appdx:LMVD}

Consider the dynamics originally proposed in \cite{leimkuhler_ensemble_2018}, Equation 6, 
\begin{equation}
\begin{split} \label{eqn:LMVD}
d\V{X}_t &= \V{B}(\V{X}_t) \V{P}_t dt,\\
d\V{P}_t &= -\V{B}(\V{X}_t)^T \nabla V(\V{X}_t) dt + kT \text{div}(\V{B}^T)(\V{X}_t) dt - \gamma \V{P}_t dt + \sqrt{2 \gamma kT} d\V{W}_t,
\end{split}
\end{equation}
where $\V{B}(\V{X})$ is a positive definite matrix, $\gamma>0$ is a friction parameter and $\V{P}\in\mathbb{R}^n$ denotes the instantaneous system momentum. It is straightforward to check that these dynamics preserve the canonical distribution 
\[
\rho(\V{X},\V{P}) \propto \exp\left(-V(\V{X})/kT-\|\V{P}\|^2/2kT\right),
\]
for any positive definite matrix $\V{B}(\V{X})$, where the marginal distribution of position satisfies
\[
\int \rho(\V{X},\V{P}) d\V{P} \propto \rho(\V{X}).
\]
We now consider discretizations of \eqref{eqn:LMVD} built via splitting the SDE into three pieces denoted A, B, and O:
\[
\begin{split}
d\left[\begin{array}{c}
\V{X}_t\\ \V{P}_t
\end{array}
\right] = 
\underbrace{\left[\begin{array}{c}
\V{B}(\V{X}_t) \V{P}_t\\ \V{0}
\end{array}
\right] dt}_\text{A}
+&
\underbrace{\left[\begin{array}{c}
\V{0}\\ -\V{B}(\V{X}_t)^T \nabla V(\V{X}_t) + kT \text{div}(\V{B}^T)(\V{X}_t)
\end{array}
\right] dt}_\text{B}\\
&+
\underbrace{\left[\begin{array}{c}
\V{0}\\ -\gamma \V{P}_t dt + \sqrt{2 \gamma kT} d\V{W}_t
\end{array}
\right]}_\text{O}.
\end{split}
\]
Note that when $\V{B}$ is a constant matrix \eqref{eqn:LMVD} reduces to conventional Langevin dynamics, and the above splitting matches the pieces given in \cite{leimkuhler_rational_2013}.

Taking any of the A, B or O pieces in isolation, we may solve the implied SDE exactly (in distribution) for time $t>0$. Denoting the solution to each piece as $\phi_t(\V{X},\V{P})$, given the initial conditions at $t=0$ are $(\V{X},\V{P})$, we can write 
\[
\begin{split}
    \phi_t^{\text{A}}(\V{X},\V{P}) &= (
    \left\{\V{Y}(t) \middle| \V{Y}(0)=\V{X}, d\V{Y} = \V{B}(\V{Y})\V{P}dt\right\},
    \V{P}),\\
    \phi_t^{\text{B}}(\V{X},\V{P}) &= (\V{X},\V{P}-t \V{B}(\V{X})^T \nabla V(\V{X}) + t kT \text{div}(\V{B}^T)(\V{X})),\\
    \phi_t^{\text{O}}(\V{X},\V{P}) &= (\V{X},e^{-\gamma t}\V{P} + \sqrt{kT} \sqrt{1-e^{-2\gamma t}} \V{R}),
\end{split}
\]
where $\V{R}\sim N(\V{0},\V{I})$ is a normal random vector. As $\phi^\text{A}$ has no explicit closed form, we write the update purely as the solution to the underlying ODE.

We now consider the overdamped limit $\gamma \to \infty$ with a time step $s>0$, using the discretization scheme 
\[
(\V{X}_{n+1},\V{P}_{n+1}) := \phi_{s/2}^\text{A} \circ \phi_{s}^\text{O} \circ \phi_{s/2}^\text{A} \circ \phi_{s}^\text{B}(\V{X}_n,\V{P}_n).
\]
Writing out the resulting steps, we obtain
\[
\begin{split}
    \hat{\V{P}}_n &= \V{P}_n-s \V{B}(\V{X}_n)^T {\nabla} V(\V{X}_n) + s kT \text{div}(\V{B}^T)(\V{X}_n)\\
    \hat{\V{X}}_n &= \left\{\V{Y}(s/2) \middle| \V{Y}(0)=\V{X}_n, d\V{Y} = \V{B}(\V{Y})\hat{\V{P}}_n \right\} \\ 
    \V{P}_{n+1} &= \sqrt{kT} \V{R}_{n+1}\\
    {\V{X}}_{n+1} &= \left\{\V{Y}(s/2) \middle| \V{Y}(0)=\hat{\V{X}}_n, d\V{Y} = \V{B}(\V{Y}) {\V{P}}_{n+1} \right\}
\end{split}
\]
which we may simplify by recognizing that $\V{P}_{n} \equiv \sqrt{kT} \V{R}_{n}$.

We recover the LMVD method given in \eqref{eqn:LMVD_defn} by considering the one-dimensional case where $B(x)=\sqrt{D(x)}$ and choosing $s=\sqrt{2h}$ for a time step of $h>0$ to ensure consistency between schemes \cite{lelievre_free_2010}. 

\section{Proofs}
\label{appdx:Proofs}

This section contains proofs of all results stated in Section \ref{sec:OneDimensionTransforms} and Section \ref{sec:MultivariateTheory}.

\begin{theorem}
    \label{thm:transformationalSymmetry}
    Applying a continuous coordinate transform to a one-dimensional Brownian dynamics process results in another Brownian dynamics process with potential and diffusion function given by \eqref{eqn:lamperti1DPotentialAndD}.
\end{theorem}
\begin{proof}
Applying It\^o's Lemma with $y=y(x)$, where $x$ obeys \eqref{1Dvariable2}, we have,
\[
\begin{split}
    dy = \left(-D(x(y))\frac{dV}{dx}(x(y)) + kT \frac{dD}{dx}(x(y))\right)\frac{dy}{dx}(x(y))dt \\ + kT D(x(y)) \frac{d^2 y}{dx^2}(x(y))dt  + \sqrt{2 kT D(x(y))}\frac{dy}{dx}(x(y)) dW_t. 
\end{split}
\]
Now substituting in the transformations \eqref{eqn:lamperti1DPotentialAndD},
\[
    \begin{split}
        =& -\hat{D}\frac{d}{dx} \left( \hat{V}dt - kT\ln \left \vert \frac{dy}{dx} \right \vert \right) \left(\frac{dy}{dx}\right)^{-1}dt + kT \frac{d}{dx}\left(\hat{D}\frac{dy}{dx}^{-2} \right)\frac{dy}{dx}dt \\ & + kT \hat{D}\frac{dy}{dx}^{-2} \frac{d^2 y}{dx^2} dt + \sqrt{2 kT \hat{D}} dW_t \\ 
        =& -\hat{D}\frac{d\hat{V}}{dy}dt + kT\hat{D}\frac{d^2 y}{dx^2} \frac{dy}{dx}^{-2}dt + kT \frac{d\hat{D}}{dy}dt -2kT \hat{D}\frac{d^2y}{dx^2}\frac{dy}{dx}^{-2} + kT \hat{D}\frac{dy}{dx}^{-2} \frac{d^2 y}{dx^2} dt \\ & + \sqrt{2 kT \hat{D}} dW_t  \\
        =& -\hat{D}(y) \frac{d\hat{V}}{dy}dt + kT \frac{d\hat{D}}{dy}dt + \sqrt{2 kT \hat{D}(y)} dW_t,
    \end{split}
\]
which is Brownian motion with potential $\hat{V}(y)$ and diffusion function $\hat{D}(y)$.
\end{proof}

\noindent
\begin{theorem}
    \label{thm:LampertiPhaseSpace1D}
    In one dimension, the Lamperti-transformed process can be used to compute phase space averages through \eqref{eqn:1DLampertiErgodicTheorem}.
\end{theorem}
\begin{proof}
    We assume that the transformed process $y_t$ is ergodic and therefore satisfies
\begin{equation}
\label{ergodicLampertiProcess}
\lim_{T_{\text{sim}} \rightarrow \infty} \frac{1}{T_{\text{sim}}} \int_{t=0}^{T_{\text{sim}}} f(y_t) dt = \int_{-\infty}^{\infty} f(y) \hat{\rho}(y) dy,
\end{equation}
where $\hat{\rho}(y) = \frac{1}{\hat{Z}} \exp{\left(-\frac{-\hat{V}(y)}{kT}\right)}$ is the invariant distribution of the transformed process. Substituting in the effective potential from \eqref{eqn:lamperti1DPotentialAndD}, the right-hand side of \eqref{ergodicLampertiProcess} becomes
\[
\int_{-\infty}^{\infty} f(y) \frac{\exp(-\frac{V(x(y))}{kT})}{\hat{Z}} \sqrt{D(x(y))} dy = \int_{-\infty}^{\infty} f(y(x)) \frac{\exp(-\frac{V(x)}{kT})}{\hat{Z}} \sqrt{D(x)} \frac{dy}{dx} dx.
\]
Using the fact $\frac{dy}{dx}=\frac{1}{\sqrt{D(x)}}$, this equation simplifies to 
\begin{equation}
\label{simplifiedSubstituted}
\int_{-\infty}^{\infty} f(y(x)) \frac{\exp(-\frac{V(x)}{kT})}{\hat{Z}}dx = \frac{Z}{\hat{Z}}\int_{- \infty}^{\infty} f(y(x)) \rho(x) dx
\end{equation}
where $Z = \int_{-\infty}^{\infty} \exp\left(-\frac{V(x)}{kT}\right)dx$ and $\hat{Z} = \int_{-\infty}^{\infty} \exp\left(-\frac{\hat{V}(y)}{kT}\right)dy$ 
are the partition functions of the original and transformed processes respectively. 

But $Z=\hat{Z}$ since, by change of variables:
\[
    \begin{split}
        Z &= \int_{-\infty}^{\infty} \exp\left(-\frac{V(x)}{kT}\right) dx = \int_{-\infty}^{\infty} \exp\left(-\frac{V(x(y))}{kT}\right) \left(\frac{dy}{dx}(x(y))\right)^{-1} dy \\
        &= \int_{-\infty}^{\infty} \exp\left(-\frac{V(x(y))+kT \ln \vert \frac{dy}{dx}(x(y)) \vert}{kT}\right) dy = \int_{-\infty}^{\infty} \exp\left(-\frac{\hat{V}(y)}{kT}\right) dy \\ &= \hat{Z}.
    \end{split}
\]

\noindent
Hence, from \eqref{ergodicLampertiProcess} and \eqref{simplifiedSubstituted} we have
\[
\int_{-\infty}^{\infty} f(y(x))\rho(x)dx =  \lim_{T_{\text{sim}} \rightarrow \infty} \frac{1}{T_{\text{sim}}}\int_{t=0}^{T_{\text{sim}}} f(y_t) dt.
\]
Finally, if we redefine $f$ as $f \circ x$, then we obtain:
\begin{equation}
\label{eqn:1DLampertiErgodicTheoremAppx}
\int_{-\infty}^{\infty} f(x)\rho(x)dx =  \lim_{T_{\text{sim}} \rightarrow \infty} \frac{1}{T_{\text{sim}}} \int_{t=0}^{T_{\text{sim}}} f(x(y_t)) dt,
\end{equation}
as required.
\end{proof}

\begin{theorem}
    \label{thm:InvariantLamperti}
    In one dimension, the Lamperti-transformed invariant measure $\hat{\rho}(y)$ and the original invariant measure $\rho(x)$ are related by $\rho(x) = \hat{\rho}(x(y))\frac{dy}{dx}$.
\end{theorem}
\begin{proof}
  Set $f(x) = I(x \in [a, b])$, the indicator function on the interval $[a, b]$, in \eqref{eqn:1DLampertiErgodicTheoremAppx}. This gives
\begin{equation}
    \label{invariantDistributionTransform}
    \begin{split}
        \int_{-\infty}^{\infty} f(x)\rho(x)dx &= \int_a^b \rho(x) dx  = \lim_{T \rightarrow \infty} \frac{1}{T_{\text{sim}}}\int_{t=0}^{T_{\text{sim}}} I(x(y_t) \in [a,b]) dt \\ &= \lim_{T_{\text{sim}} \rightarrow \infty} \frac{1}{T_{\text{sim}}}\int_{t=0}^{T_{\text{sim}}} I(y_t \in [y(a),y(b)]) dt = \int_{y(a)}^{y(b)} \hat{\rho}(y) dy \\ &= \int_{a}^{b} \hat{\rho}(x(y)) \frac{dy}{dx}dx,
    \end{split}
\end{equation}
\noindent
where in the second line we re-expressed the indicator function in terms of $y_t$ and then applied the ergodic theorem for the $y_t$ process. \eqref{invariantDistributionTransform} implies
\[
\int_{a}^b \left(\rho(x) - \hat{\rho}(x(y))\frac{dy}{dx}\right) dx = 0
\]
which, from the arbitrariness of the constants $a$ and $b$, proves the result.
\end{proof}

\begin{theorem}
    \label{thm:TimeRescaling1D}
    Applying a time-rescaling to a one-dimensional Brownian dynamics process results in another Brownian dynamics process with potential and diffusion given by \eqref{eqn:effectivePotentialTimeRescaling}.
\end{theorem}
\begin{proof}
Applying a version of the time rescaling appearing in equation \eqref{timeRescaledEqnsBeforeSub} to one-dimensional Brownian dynamics we arrive at, 
\begin{equation}
    \label{timeRescaled1DBrownianDynamicsBeforeSub}
    dx_\tau = -g(x)D(x)\frac{dV}{dx}d\tau + kT g(x) \frac{dD(x)}{dx} d\tau + \sqrt{2 kT g(x) D(x)}dW_\tau.
\end{equation}
\noindent
Inserting the identities from equation \eqref{eqn:effectivePotentialTimeRescaling} into \eqref{timeRescaled1DBrownianDynamicsBeforeSub} we obtain
\[
    \label{proof1DTimeRescaling}
    \begin{split}
        dx_\tau &= - \hat{D}(x)\frac{d}{dx}\left( \hat{V}(x) - kT \ln{g(x)}\right)d\tau + kTg(x) \frac{d}{dx}\left(\frac{\hat{D}(x)}{g(x)}\right) d\tau + \sqrt{2kT \hat{D}(x)}dW_\tau \\
        &= - \hat{D}(x)\frac{d\hat{V}}{dx}d\tau + kT \hat{D}(x) \frac{g'(x)}{g(x)}d\tau + kT\frac{d\hat{D}(x)}{dx} d\tau - kT\hat{D}(x)\frac{g'(x)}{g(x)} d\tau + \sqrt{2kT \hat{D}(x)}dW_\tau \\
        &= - \hat{D}(x)\frac{d\hat{V}}{dx}d\tau + kT \frac{\hat{dD}(x)}{dx} d\tau + \sqrt{2kT \hat{D}(x)}dW_\tau,
    \end{split}
\]
which is a transformed version of the original one-dimensional Brownian dynamics but in an effective potential $\hat{V}(x)$ and a rescaled diffusion coefficient $\hat{D}(x)$, as required.
\end{proof}

\begin{theorem}
    \label{thm:TimeRescalingPhaseSpace1D}
    In one dimension, the time-rescaled process can be used to compute phase space averages through  \eqref{reweightingTimeRescaling0}.
\end{theorem}
\begin{proof}
From the  ergodic theorem applied to the original process,
\[
\label{ensembleAverageMethod2}
\begin{split}
\int_{-\infty}^{\infty} f(x) \rho(x) dx &= \lim_{T_{\text{sim}} \xrightarrow{} \infty} \frac{1}{T_{\text{sim}}}\int_{t=0}^{T_{\text{sim}}} f(x_t)dt.
\end{split}
\]
Changing variables $t \rightarrow \tau$ in the integration, the right-hand side becomes,
\[
\lim_{T_{\text{sim}} \xrightarrow{} \infty} \frac{1}{T_{\text{sim}}}\int_{\tau=0}^{\tau(T_{\text{sim}})}f(x_\tau) \frac{dt}{d\tau}d\tau = \lim_{T_{\text{sim}} \xrightarrow{} \infty} \frac{1}{T_{\text{sim}}} \int_{\tau=0}^{\tau(T_{\text{sim}})} f(x_\tau) g(x_\tau) d\tau.
\]
Writing $\mathcal{T}_{\text{sim}} \vcentcolon= \tau(T_{\text{sim}})$, this can be alternatively written as
\[
\lim_{\mathcal{T}_{\text{sim}} \xrightarrow{} \infty} \frac{1}{t(\mathcal{T}_{\text{sim}})} \int_{\tau=0}^{\mathcal{T}_{\text{sim}}} f(x_\tau) g(x_\tau) d\tau.
\]
Finally, we note that by integrating $\frac{dt}{d\tau}=g(x)$ between $0$ and $\mathcal{T}_{\text{sim}}$ we can obtain an expression for $t(\mathcal{T}_{\text{sim}})$,
\[
\label{tau_t_time_conversion}
t(\mathcal{T}_{\text{sim}}) = \int_{\tau=0}^{\mathcal{T}_{\text{sim}}} g(x_\tau) d\tau.
\]
Substituting this into the above equation completes the proof.
\end{proof}

\begin{theorem}
\label{thm:MultivariateLamperti}
A Lamperti-transformed process with original $\V{D}$ matrix of the form $\V{D}(\V{X})_{ij} = D_i(X_i)R_{ij}$ is an instance of Brownian dynamics if and only if $\mathbf{R}$ is diagonal. Further, when $R_{ij} = \delta_{ij}$, the effective potential is given by \eqref{eqn:LampertiMultivariateEffectivePotential}.
\end{theorem}
\begin{proof}
The stated transformation is a multivariate Lamperti transform \eqref{multivariateTransformedLampertiProcess} with 
\[
    f(\V{X}) = - \V{D}(\V{X})\V{D}(\V{X})^T \nabla V(\V{X}) + kT \text{div}(\V{D}\V{D}^T)(\V{X}), \quad \sigma(\V{X}) = \sqrt{2kT}\V{D}(\V{X}).
\]
The transformed process therefore satisfies
\[
    \begin{split}
    dY_{i,t} = &\sum_{j=1}^n R^{-1}_{ij}\sqrt{2kT} \left(\frac{-\sum_{k=1}^n(\V{D}\V{D}^T)_{jk}\partial_k V}{\sqrt{2kT}D_j} + \frac{kT \sum_{k=1}^n\partial_k(\V{D}\V{D}^T)_{jk}}{\sqrt{2kT}D_j}-\frac{1}{2}\sqrt{2kT}\partial_j D_j\right)dt \\ & \hspace{11cm} + \sqrt{2kT}dW_i,
    \end{split}
\]
where $\partial_j \vcentcolon= \frac{\partial}{\partial X_j}$ and $V$, $\V{D}$ and $D_j$ are functions of $\V{Y}_t$ through the relations
\[
    V(\V{X}_t) = V(\phi^{-1}(\V{RY}_t)), \quad \V{D}(\V{X}_t) = \V{D}(\phi^{-1}(\V{RY}_t)), \quad D(X_{j,t}) = D(\phi^{-1}_j((\V{RY})_{j,t})).
\]
Substituting $\V{D}(\V{X})_{ij} = D_i(X_i)R_{ij}$, this becomes
\[
\begin{split}
dY_{i,t} = \sum_{j,k,l=1}^n R_{ij}^{-1} \left(\frac{- R_{jk}R_{kl}D_jD_k \partial_k V}{D_j}+kT\frac{R_{jl}R_{kl}\partial_k (D_j D_k)}{D_j}\right)dt-kT \sum_{j=1}^n R_{ij}^{-1} \partial_j D_jdt \\ + \sqrt{2kT}dW_i.
\end{split}
\]
Expanding,
\[
\begin{split}
dY_{i,t} = \sum_{k=1}^n -R_{ki}D_k \partial_k V dt + kT\left(\sum_{j,k,l=1}^n R_{ij}^{-1}R_{jl}R_{kl} \left(\partial_k D_j \frac{D_k}{D_j} + \partial_k D_k\right) - \sum_{j=1}^n R_{ij}^{-1}\partial_j D_j\right)dt \\ + \sqrt{2kT}dW_i.
\end{split}
\]
Noting that $\partial_k D_j = \delta_{kj} \partial_j D_j$, this becomes
\[
\begin{split}
dY_{i,t} = \sum_{k=1}^n -R_{ki}D_k \partial_k V dt + kT \left( \sum_{j,l} R_{ij}^{-1}R_{jl}R_{jl}\partial_k D_k + \sum_{k} R_{ki} \partial_k D_k - \sum_{j=1}^n R_{ij}^{-1} \partial_j D_j\right)dt \\ + \sqrt{2kT}dW_i.
\end{split}
\]
Changing variables so that the derivatives are with respect to $\V{Y}$ we get 
\[
\frac{\partial}{\partial X_k} = \sum_{l=1}^n \frac{\partial Y_l}{\partial X_k} \frac{\partial}{\partial Y_l} = \sum_{l=1}^n \frac{R^{-1}_{lk}}{D_k(X_k)} \frac{\partial}{\partial Y_l},
\]
and the transformed equation becomes
\[
\begin{split}
dY_{i,t} = -\partial_i V dt + kT \sum_{k=1}^n \left( \left( \sum_{l=1}^n R_{ik}^{-1}R_{kl}R_{kl}R_{kk}^{-1}\right) + R_{ki}R_{kk}^{-1} + R_{ik}^{-1}R_{kk}^{-1}\right) \nabla_{Y_k} 
 \ln D_k dt \\ + \sqrt{2kT}dW_i,
\end{split}
\]
or equivalently:
\[
dY_{i,t} = -\partial_i V dt + kT \sum_{k=1}^n M_{ik} \nabla_{Y_k} 
 \ln D_k dt + \sqrt{2kT}dW_i,
\]
where $\V{M}$ is the matrix defined in the theorem statement. 
Note that only if the matrix $\V{R}$ is diagonal (and therefore, $M_{ij} = \delta_{ij}$) is it possible to express the drift term as a gradient of an effective potential, given by
\begin{equation}
\label{effectivePotentialGeneralLamperti}
\hat{V}(\V{Y}) = V(\phi^{-1}(\V{RY})) - kT \sum_{i=1}^n \ln D_i (\phi^{-1}_i (\V{RY}_{i})).
\end{equation}
In particular, if we set $R_{ij} = \delta_{ij}$, then we identify the transformed process as constant-diffusion Brownian dynamics with an effective potential
\begin{equation}
\label{eqn:effectivePotentialLampertiI}
\hat{V}(\V{Y}) = V(\phi^{-1}(\V{Y})) - kT \sum_{k=1}^n \ln D_k (\phi^{-1}_k (Y_{k})).
\end{equation}
\begin{remark}
    The functions $D_i$ can be arbitrarily scaled in such a manner that for all $i$, $R_{ii} = 1$ in equation \eqref{effectivePotentialGeneralLamperti}. The transformed process then becomes equivalent to the case $\V{R}=\V{I}$, as discussed in Section \ref{sec:multivariateLamperti}.
\end{remark}
\end{proof}

\begin{theorem}
    \label{thm:LampertiErgodic}
    A Lamperti-transformed process with original $\V{D}$ matrix of the form $\V{D}(\V{X}) = D_i(X_i)\delta_{ij}$ can be used to compute phase-space averages through \eqref{eqn:NDLampertiErgodicTheorem}.
\end{theorem}
\begin{proof}
We assume that the effective potential \eqref{eqn:effectivePotentialLampertiI} is such that geometric ergodicity holds for $\V{Y}_t$. Then, by applying the ergodic theorem to the transformed process, we obtain: 
\[
    \lim_{T_{\text{sim}} \rightarrow \infty} \frac{1}{T_{\text{sim}}} \int_{t=0}^{T_{\text{sim}}} f(\V{Y}_t) dt = \int_{\mathbb{R}^n} f(\V{Y})\hat{\rho}(\V{Y}) d\V{Y}.
\]
Substituting in the effective potential, we have:
\[
    = \int_{\mathbb{R}^n} f(\V{Y})\frac{1}{\hat{Z}}\exp{\left(-\frac{V(\phi^{-1}(\V{Y}))}{kT} \right)} \prod_{i=1}^n \left(D_i (\phi^{-1}_i(Y_{i,t})) \right)d\V{Y}
\]
Next, we change variables from $\V{Y}$ to $\V{X}$. The Jacobian factor is given by:
\[
    J = \left|\frac{d\V{Y}}{d\V{X}}\right| = \left| \frac{1}{D_i(X_i)} \delta_{ij}\right| = \prod_{i=1}^n \frac{1}{D_i(X_i)},
\]
which exactly cancels with the diffusion coefficients in the integral, and we have
\[
    = \int_{\mathbb{R}^n} f(\phi(\V{X}))\frac{1}{\hat{Z}}\exp{\left(-\frac{V(\V{X})}{kT} \right)}d\V{X} = \frac{Z}{\hat{Z}}\int_{\mathbb{R}^n} f(\phi(\V{X}))\rho(\V{X})d\V{X}. 
\]
Choosing $f(\V{Y})=1$ leads to $\hat{Z} = Z$, hence
\[
\lim_{T_{\text{sim}} \rightarrow \infty} \frac{1}{T_{\text{sim}}} \int_{t=0}^{T_{\text{sim}}} f(\V{Y}_t) dt = \int_{\mathbb{R}^n} f(\phi(\V{X}))\rho(\V{X})d\V{X}.
\]
Finally, if we redefine $f$ as $f \circ \phi^{-1}$, then we obtain:
\[
\lim_{T_{\text{sim}} \rightarrow \infty} \frac{1}{T_{\text{sim}}} \int_{t=0}^{T_{\text{sim}}} f(\phi^{-1}(\V{Y}_t)) dt = \int_{\mathbb{R}^n} f(\V{X})\rho(\V{X})d\V{X},
\]
as required.
\end{proof}

\begin{theorem}
    \label{thm:TimeRescalingMD}
    The effective potential of a time-rescaled Brownian dynamics process with original $\V{D}$ matrix $\V{D}(\V{X}) = D(\V{X})\V{R}$ is given by \eqref{eqn:NDEffectivePotentialTimeRescaling}.
\end{theorem}
\begin{proof}
The time-rescaling transform follows \eqref{eqn:timeRescalingFinal}, \eqref{eqn:timeRescalingLinearTransform} with $f(\V{X})=-\V{D}(\V{X})\V{D}(\V{X})^T \nabla V(\V{X}) + kT \text{div}(\V{D}(\V{X})\V{D}(\V{X})^T)$, which gives (we transform to constant diffusion $\sqrt{2kT}$):
\[
d\V{Y}_\tau = - \V{R}^{-1}\frac{\V{D}\V{D}^T\nabla_{\V{X}} V-kT\text{div}(\V{D}\V{D}^T)}{D^2}dt + \sqrt{2kT}d\V{W}_\tau.
\]
Here, $V$, $\V{D}$ and $D$ are functions of $\V{Y}_\tau$ through the relations:
\[
V(\V{X}_\tau) = V(\V{R}\V{Y}_\tau), \quad \V{D}(\V{X}_\tau) = \V{D}(\V{R}\V{Y}_\tau), \quad D(\V{X}_\tau) = D(\V{R}\V{Y}_\tau).
\]
Substituting $\V{D}(\V{X})=D(\V{X})\V{R}$, in components this becomes
\[
dY_{i, \tau} =  - \frac{\sum_{j} D^2R_{ji}\partial_j V - kT\sum_j R_{ji} \partial_j (D^2) }{D^2}dt + \sqrt{2kT}dW_{i,\tau},
\]
which simplifies to 
\[
dY_{i, \tau} =  \sum_j R_{ji} \left(-\partial_j V + 2kT \partial_j \ln D\right)dt + \sqrt{2kT}dW_{i,\tau}.
\]
Changing variables so that the derivatives are with respect to $\V{Y}$ we get 
\[
\frac{\partial}{\partial X_i} = \sum_{j=1}^n \frac{\partial Y_j}{\partial X_i} \frac{\partial}{\partial Y_j} = \sum_{j=1}^n R_{ji}^{-1} \frac{\partial}{\partial Y_j}.
\]
The $\V{R}$ matrix then cancels with its inverse, and the dynamics now reads
\[
d\V{Y}_{\tau} = \left(-\nabla_{\V{Y}} V(\V{RY}) + 2kT \nabla_{\V{Y}} \ln D(\V{RY}) \right)dt + \sqrt{2kT}d\V{W}_\tau,
\]
which is constant-diffusion Brownian dynamics in an effective potential $\hat{V}(\V{Y})$ given by
\[
\hat{V}(\V{Y}) = V(\V{RY})- 2kT \ln D(\V{RY}).
\]
This completes the proof.
\end{proof}

\begin{theorem}
    \label{thm:TimeRescalingMDErgodic}
    A time-rescaled process with original $\V{D}$ matrix of the form $\V{D}(\V{X}) = D(\V{X})\V{R}$ can be used to compute phase-space averages through \eqref{eqn:NDtimeErgodicTheorem}.
\end{theorem}
\begin{proof}
We begin with the  ergodic theorem of the original process, which states
\[
\begin{split}
\int_{\mathbb{R}^n} f(\V{X}) \rho(\V{X}) d\V{X} &= \lim_{T_{\text{sim}} \xrightarrow{} \infty} \frac{1}{T_{\text{sim}}}\int_{t=0}^{T_{\text{sim}}} f(\V{X}_t)dt.
\end{split}
\]
To express this in terms of the time-rescaled process, we change the variable $t \to \tau$ in the integration, resulting in:
\[
\lim_{T_{\text{sim}} \xrightarrow{} \infty} \frac{1}{T_{\text{sim}}}\int_{\tau=0}^{\tau(T_{\text{sim}})}f(\V{X}_\tau) \frac{dt}{d\tau}d\tau = \lim_{T_{\text{sim}} \xrightarrow{} \infty} \frac{1}{T_{\text{sim}}} \int_{\tau=0}^{\tau(T_{\text{sim}})} f(\V{X}_\tau) g(\V{X}_\tau) d\tau.
\]
Writing $\mathcal{T}_{\text{sim}} := \tau(T_{\text{sim}})$, this can alternatively be written as
\begin{equation}
\label{partialProofNDTimeRescalingErgodicity}
    \lim_{\mathcal{T}_{\text{sim}} \xrightarrow{} \infty} \frac{1}{t(\mathcal{T}_{\text{sim}})} \int_{\tau=0}^{\mathcal{T}_{\text{sim}}} f(\V{X}_\tau) g(\V{X}_\tau) d\tau,
\end{equation}
where $g(\V{X}_\tau) = 1/D^2(\V{X})$ by the definition of time rescaling.
Next, we integrate $\frac{dt}{d\tau}=g(\V{X})$ from $0$ to $\mathcal{T}_{\text{sim}}$ to obtain an expression for $t(\mathcal{T}_{\text{sim}})$,
\begin{equation}
\label{tTauTransformND}
   t(\mathcal{T}_{\text{sim}}) = \int_{\tau=0}^{\mathcal{T}_{\text{sim}}} g(\V{X}_\tau) d\tau,
\end{equation}
Substituting equation \eqref{tTauTransformND} and the relation $\mathbf{X}_\tau = \mathbf{R}\mathbf{Y}_\tau$ into equation \eqref{partialProofNDTimeRescalingErgodicity}, we have:
\[
    \int_{\mathbb{R}^n} f(\V{X}) \rho(\V{X}) d\V{X} = \lim_{\mathcal{T}_{\text{sim}} \rightarrow \infty} \frac{\int_{\tau=0}^{\mathcal{T}_{\text{sim}}} f(\V{R}\V{Y}_\tau) g(\V{R}\V{Y}_\tau) d\tau}{\int_{\tau=0}^{\mathcal{T}_{\text{sim}}} g(\V{R}\V{Y}_\tau) d\tau},
\]
as required.
\end{proof}

\begin{theorem}
    \label{thm:NoDRDDiffusion}
    Performing a time rescaling followed by a Lamperti transform to a multivariate Brownian dynamics process with $\V{D}$ matrix $\V{D}(\V{X}) = \V{D}^{(1)}(\V{X})\V{R}\V{D}^{(2)}(\V{X})$, where $\V{R}$ is not diagonal, results in a constant-diffusion process that is not Brownian dynamics.
\end{theorem}
\begin{proof}
    First, consider the time-rescaled process $\V{X}_{\tau}$, where $\frac{dt}{d\tau}=1/D^2(\V{X})$, which obeys the dynamics
\[
d\V{X}_\tau = -\frac{\V{D}\V{D}^T\nabla_{\V{X}} V-kT\text{div}(\V{D}\V{D}^T)}{D^2}dt + \sqrt{2kT}\V{R}\V{D}^{(2)}d\V{W}_\tau.
\]
Defining a transformed process $\V{Y}_{\tau} = \V{R}^{-1}\V{X}_\tau$, then applying the multidimensional It\^o formula gives 
\[
d\V{Y}_\tau = -\V{R}^{-1}\left(\frac{\V{D}\V{D}^T\nabla_{\V{X}} V\vert_{\V{R}\V{Y}_\tau}-kT\text{div}(\V{D}\V{D}^T)\vert_{\V{R}\V{Y}_\tau}}{D^2}\right)dt + \sqrt{2kT}\V{D}^{(2)}d\V{W}_\tau,
\]
where
\[
    V = V(\V{RY}_\tau), \quad \V{D} = \V{D}(\V{RY}_\tau), \quad \V{D}^{(2)}=\V{D}^{(2)}(\V{R}\V{Y}_\tau), \quad D = D(\V{RY}_{\tau}).
\]

\noindent
Finally, we apply a Lamperti transform to remove the noise dependence on $\V{D}^{(2)}$. The transformed process $\V{Z}_{\tau}= \phi(\V{Y}_\tau)$ then satisfies (using Einstein summation convention for sums over repeated indices)
\[
d\V{Z}_{i, \tau} = - R^{-1}_{ij}\left( \frac{D_{jk}D_{lk} \partial_l V\vert_{\V{R}\phi^{-1}(\V{Z}_\tau)} - kT \partial_l (D_{jk}D_{lk})\vert_{\V{R}\phi^{-1}(\V{Z}_\tau)}}{D^2D_i}\right)dt + \sqrt{2kT}dW_{i,\tau}.
\]
Changing variables,
\[
\frac{\partial}{\partial X_k} = \sum_{l=1}^n \frac{\partial Z_l}{\partial X_k} \frac{\partial}{\partial Z_l} = \sum_{l=1}^n \frac{R^{-1}_{lk}}{D_k(X_k)} \frac{\partial}{\partial Z_l}
\]
and substituting $D_{ij} = \delta_{ik}\delta_{lj}R_{kl}DD_l$ this becomes
\[
\begin{split}
d\V{Z}_{i, \tau} = - R^{-1}_{ij}\Bigg( \frac{\delta_{jm}\delta_{nk} \delta_{lp}\delta_{qk}R_{mn}R_{pq}D^2 D_nD_q R^{-1}_{rl} \nabla_{Y_r} V}{D^2D_iD_l} \\ \frac{- kT R^{-1}_{rl}\nabla_{Y_r} (\delta_{jm}\delta_{nk} \delta_{lp}\delta_{qk}R_{mn}R_{pq}D^2 D_nD_q)}{D^2D_iD_l}\Bigg)dt + \sqrt{2kT}dW_{i,\tau},
\end{split}
\]
which simplifies to 
\[
d\V{Z}_{i, \tau} = - R^{-1}_{ij}\left( \frac{ R_{jk}R_{lk}D^2 D^2_k R^{-1}_{rl} \nabla_{Y_r} V - kT R^{-1}_{rl}\nabla_{Y_r} (R_{jk}R_{lk}D^2 D^2_k)}{D^2D_iD_l}\right)dt + \sqrt{2kT}dW_{i,\tau},
\]
\[
d\V{Z}_{i, \tau} \stackrel{\text{(no sum } i\text{)}}{=} - \left( \frac{ R_{li}D^2 D^2_i R^{-1}_{rl} \nabla_{Y_r} V - kT R^{-1}_{rl}R_{li}\nabla_{Y_r} (D^2 D^2_i)}{D^2D_iD_l}\right)dt + \sqrt{2kT}dW_{i,\tau}.
\]
Expanding this expression does not lead to a great simplification of terms. In particular, the non-vanishing of the $\V{R}$ matrix elements in the $dt$ term means that the drift term cannot be written as a gradient of a potential energy, hence this is not Brownian dynamics.
\end{proof}

\begin{theorem}
    \label{thm:DDDiffusion}
    Consider a multivariate Brownian dynamics process $\V{X}_t$ following \eqref{multivariateBrownianDiffusion}, where the diffusion tensor $\V{D}$ is defined as 
    \[
    \V{D}(\V{X}) = \V{D}^{(1)}(\V{X})\V{D}^{(2)}(\V{X}),
    \]
    where $\V{D}^{(1)}$ and $\V{D}^{(2)}$ are given by \eqref{eqn:Dforms}. Then the transformed process $\V{Y}_{\tau} = \phi(\V{X}_\tau)$, resulting from a time rescaling where $\frac{dt}{d\tau}=g(\V{X})\vcentcolon=1/D^2(\V{X})$, followed by a Lamperti transform given by
    \[
    Y_{i,\tau} = \sqrt{2kT}\int_{x_0}^{X_{i,\tau}}\frac{1}{D_i(x)}dx \vcentcolon= \sqrt{2kT} \phi_i(X_{i, \tau}),
    \]
    satisfies the constant-diffusion Brownian dynamics process:
    \[
    dY_{i,\tau} = - \nabla_{Y_i}\hat{V}(\V{Y})dt + \sqrt{2kT}dW_i,
    \]
    where $\hat{V}(\V{Y})$ is the effective potential defined as
    \[
    \hat{V}(\V{Y}) = V(\phi^{-1}(\V{Y}))-2kT \ln D(\V{Y}) - kT \sum_{i=1}^n \ln D_k (\phi^{-1}_k(Y_{k,\tau})).
    \]
\end{theorem}
\begin{proof}
    The time-rescaling transformation gives 
    \begin{equation}
    d\V{X}_\tau = - \frac{\V{D}\V{D}^T \nabla_{\V{X}}V - kT \text{div}(\V{D}\V{D}^T)}{D^2}dt + \sqrt{2kT}\V{D}^{(2)}(\V{X})d\V{W}_\tau.
    \end{equation}
    Applying the Lamperti transform then gives
    \[
    d\V{Y}_{i, \tau} = \sqrt{2kT}\left(- \frac{D_{ij}D_{kj}\partial_{k}V -kT \partial_k (D_{ij}D_{kj})}{\sqrt{2kT}D^2 D_i} - \frac{1}{2}\sqrt{2kT}\partial_i D_i\right)dt + \sqrt{2kT}dW_i,  
    \]
    where $V(\V{X}_\tau)$, $\V{D}(\V{X}_\tau)$, $D(\V{X}_\tau)$ and $\V{D}_i(\V{X}_\tau)$ are functions of $\V{Y}_\tau$ through the relation $\V{X}_\tau = \phi^{-1}(\V{Y}_\tau)$.
    Since
    \[
    D_{ij}(\V{X}) = \sum_{k=1}^n D^{(1)}_{ik}(\V{X})D^{(2)}_{kj}(\V{X}) = \sum_{k=1}^n \delta_{ik} \delta_{kj} D(\V{X}) D_k(X_k) = D(\V{X})D_i(X_i),
    \]
    this becomes
    \[
        dY_{i, \tau} = - D_i \partial_i V dt + kT \frac{\partial_i(D^2 D_i^2)}{D^2 D_i} dt - kT \partial_i D_i dt + \sqrt{2kT}dW_i
    \]
    which simplifies to
    \[
        dY_{i, \tau} = - D_i \partial_i V dt + 2kT \frac{D_i}{D} \partial_i D dt + kT \partial_i D_i dt + \sqrt{2kT}dW_i.
    \]
    Changing variables so that the derivatives are with respect to $\V{Y}$ we get
    \[
    \frac{\partial}{\partial X_k} = \sum_{l=1}^n \frac{\partial Y_l}{\partial X_k} \frac{\partial}{\partial Y_l} = \sum_{l=1}^n \frac{\delta_{lk}}{D_k(X_k)} \frac{\partial}{\partial Y_l}=\frac{1}{D_k}\frac{\partial}{\partial Y_k},
    \]
    and the transformed equation becomes
    \[
    dY_{i, \tau} = - \nabla_{Y_i} V dt + 2kT \nabla_{Y_i} \ln{D} dt + kT \nabla_{Y_i} \ln{D_i} dt + \sqrt{2kT}dW_i,
    \]
    which we identify as Brownian motion in an effective potential
    \vspace{-0.1cm}
    \[
    \hat{V}(\V{Y}) = V(\phi^{-1}(\V{Y})) - 2kT \ln{D(\phi^{-1}(\V{Y}))} - kT \sum_{i=1}^n \ln{ D_i(\phi^{-1}_i(Y_i)},
    \]
    \vspace{-0.1cm}
    as required.
\end{proof}

\appendix

\end{document}